\documentclass[a4paper,10pt]{article}
\usepackage{cmap}
\usepackage[T1]{fontenc}
\usepackage[utf8]{inputenc}
\usepackage{amsmath,amsfonts,amssymb,amsthm}
\usepackage{xcolor,graphicx}
\usepackage{a4wide}
\usepackage[affil-it]{authblk}

\allowdisplaybreaks
\usepackage{mathtools}

\usepackage{secdot}

\usepackage{epstopdf}
\usepackage[labelsep=period]{caption}

\usepackage{enumitem}

\numberwithin{equation}{section}

\let\OLDthebibliography\thebibliography
\renewcommand\thebibliography[1]{
\OLDthebibliography{#1}
\setlength{\parskip}{1pt}
\setlength{\itemsep}{1pt plus 0.3ex}
}

\usepackage[colorlinks=true,linktocpage,pdfpagelabels,
bookmarksnumbered,bookmarksopen]{hyperref}
\definecolor{ForestGreen}{rgb}{0.1,0.6,0.05}
\definecolor{EgyptBlue}{rgb}{0.063,0.1,0.6}
\hypersetup{
colorlinks=true,
linkcolor=EgyptBlue,         
citecolor=ForestGreen,
urlcolor=olive
}

\usepackage[hyperpageref]{backref}

\usepackage{accents}

\def\wolp{\accentset{\circ}{W}_p^{1}}
\def\wol2{\accentset{\circ}{W}_2^{1}}

\newtheorem{thm}{Theorem}[section]
\newtheorem{lemma}[thm]{Lemma}
\newtheorem{proposition}[thm]{Proposition}

\theoremstyle{definition}
\newtheorem{remark}[thm]{Remark}

\title{
\vspace*{-0.5cm}
Second-order derivative of domain-dependent functionals along Nehari manifold trajectories
\\ \medskip}

\author[1]{Vladimir Bobkov\thanks{E-mail: \texttt{bobkov@kma.zcu.cz}}}
\affil[1]{{\small Department of Mathematics and NTIS, Faculty of Applied Sciences, University of West Bohemia, Univerzitn\'i 8, 301 00 Plze\v{n}, Czech Republic. Institute of Mathematics, Ufa Federal Research Centre, RAS, Chernyshevsky str.~112, 450008 Ufa, Russia}}
\author[2]{Sergey Kolonitskii\thanks{E-mail: \texttt{sergey.kolonitskii@gmail.com}}}
\affil[2]{{\small Saint Petersburg Electrotechnical University "LETI", 5 Professora Popova st., St.~Petersburg, 197376 Russia}}
\date{}

\begin{document}
\maketitle

\begin{abstract}
Assume that a family of domain-dependent functionals $E_{\Omega_t}$ possesses a corresponding family of least energy critical points $u_t$ which can be found as (possibly nonunique) minimizers of $E_{\Omega_t}$ over the associated Nehari manifold $\mathcal{N}(\Omega_t)$. 
We obtain a formula for the second-order derivative of $E_{\Omega_t}$ with respect to $t$ along Nehari manifold trajectories of the form $\alpha_t(u_0(\Phi_t^{-1}(y)) + t v (\Phi_t^{-1}(y)))$, $y \in \Omega_t$, where $\Phi_t$ is a diffeomorphism such that $\Phi_t(\Omega_0) = \Omega_t$, $\alpha_t \in \mathbb{R}$ is a $\mathcal{N}(\Omega_t)$-normalization coefficient, and $v$ is a corrector function whose choice is fairly general. 
Since $E_{\Omega_t}[u_t]$ is not necessarily twice differentiable with respect to $t$ due to the possible nonuniqueness of $u_t$, the obtained formula represents an upper bound for the corresponding second superdifferential, thereby providing a convenient way to study various domain optimization problems related to $E_{\Omega_t}$. 
An analogous formula is also obtained for the first eigenvalue of the $p$-Laplacian.
As an application of our results, we investigate the behaviour of the first eigenvalue of the Laplacian with respect to particular perturbations of rectangles.  
	

\end{abstract}

\section{Introduction}\label{sec:introdution}
To outline an idea of the paper, let us start with a discussion of the model Lane-Emden problem
\begin{equation}\label{Dm}
\left\{
\begin{aligned}
-\Delta_p u &= |u|^{q-2} u 
&&\text{in } \Omega, \\
u &= 0  &&\text{on } \partial \Omega,
\end{aligned}
\right.
\end{equation}	
where $\Delta_p u := \text{div}\left(|D u|^{p-2} D u \right)$ stands for 	 the $p$-Laplacian, $2 \leq p < q < p^*$, and $\Omega \subset \mathbb{R}^N$ is a bounded domain with the boundary $\partial \Omega$, $N \geq 2$.
Here $p^* = \frac{N p}{N-p}$ for $p < N$, and $p^* = +\infty$ for $p \geq N$.
It is well-known (see, e.g., \cite{peral}) that \eqref{Dm} has infinitely many (weak) solutions, among which we will be interested in the so-called \textit{least energy solutions} (also known as \textit{ground states}). Such solutions can be defined as minimizers of the problem
$$
\mu(\Omega) := \inf_{w \in \mathcal{N}(\Omega)} E_\Omega[w]
$$
(see, for instance, \cite{GrumiauParini}), 
where $E_\Omega: \wolp(\Omega) \mapsto \mathbb{R}$ is the energy functional associated with \eqref{Dm}:
$$
E_\Omega[w] = \frac{1}{p}\int_\Omega |D w|^p \,dx - \frac{1}{q}\int_\Omega |w|^q \,dx,
$$
and $\mathcal{N}(\Omega)$ is the corresponding Nehari manifold:
$$
\mathcal{N}(\Omega) = \left\{
w \in \wolp(\Omega) \setminus \{0\}:~ E_\Omega'[w] w = 0
\right\}.
$$

A natural class of optimization problems related to \eqref{Dm} consists in optimizing the value of $\mu(\Omega)$ over a set of admissible domains. For instance, the generalized Faber-Krahn inequality (see, e.g., \cite{Brasco}) can be  employed to show that $\mu(\Omega) \geq \mu(B)$ if $B$ is a ball of the same volume as $\Omega$. Moreover, the equality $\mu(\Omega) = \mu(B)$ holds if and only if $\Omega = B$. 
That is, $B$ is a global minimizer for $\mu(\Omega)$ among the set of all domains with equal volume.  
Somewhat opposite situation occurs if we consider a class of spherical shells $\Omega_t = B_R \setminus \overline{B}_r$, where $R>r$ and $t \in [0, R-r)$ is the distance between centres of the balls $B_R$, $B_r$ of radius $R, r$, respectively. In this case, $\mu(\Omega_0) \geq \mu(\Omega_t)$ and $\mu(\Omega_t)$ strictly decreases with respect to $t$, see \cite{BK}.
That is, the concentric spherical shell $\Omega_0$ is the global maximizer for $\mu(\Omega_t)$ among $t \in [0,R-r)$. We refer the reader to \cite{abs,ak} for the analogous result for the first eigenvalue of the $p$-Laplacian and for relevant references. 

In the proof of the latter optimization result, the following Hadamard-type estimate for $\mu(\Omega)$ was used. 
Consider a smooth perturbation $\Phi_t(\Omega)$ of the domain $\Omega$ driven by a family of diffeomorphisms
\begin{equation*}
\Phi_t(x) = x + t R(x), 
\quad 
R \in C^1(\mathbb{R}^N, \mathbb{R}^N).
\end{equation*}
If $u$ is a minimizer of $\mu(\Omega)$ and $\alpha_t \in \mathbb{R}$ is chosen in such a way that $\alpha_t u(\Phi_t^{-1}(\cdot)) \in \mathcal{N}(\Phi_t(\Omega))$, then 
\begin{equation}
\label{eq:had0}
\limsup_{t \to 0+} \frac{\mu(\Phi_t(\Omega)) - \mu(\Omega)}{t} 
\leq 
\left.
\frac{\partial E_{\Phi_t(\Omega)}(\alpha_t u(\Phi_t^{-1}(\cdot)))}{\partial t} \right|_{t=0}= - \frac{p-1}{p} \int_{\partial \Omega} 
\left| \frac{\partial u}{\partial n} 
\right|^p 
\left<R, n\right> \,d\sigma,
\end{equation}
provided $\partial \Omega$ is sufficiently regular, where $n$ is the outward unit normal vector to $\partial \Omega$, see \cite[Theorem 1.1]{BK}.
Notice that $\mu(\Phi_t(\Omega))$ is continuous but can be nondifferentiable with respect to $t$ since the corresponding minimizer is not necessarily unique (see \cite{coffman,Kolon,Nazar2} and a discussion in \cite[Remark 3.5]{BK}). Hence, in general, only estimates of $\mu(\Phi_t(\Omega))$ through the finite differences as in \eqref{eq:had0} are possible. 

However, in a variety of applications, consideration of the first-order approximation of $\mu(\Omega)$ does not bring a sufficient information to obtain an optimality or stability of the domain, and higher-order approximations have to be studied.
The main aim of the present paper is to provide an upper estimate for $\mu(\Phi_t(\Omega))$ in terms of the second-order derivative of $E_{\Phi_t(\Omega)}$ with respect to $t$ along trajectories of the form $\alpha_t(u(\Phi_t^{-1}(y)) + t v (\Phi_t^{-1}(y)))$, $y \in \Phi_t(\Omega)$, where $\alpha_t \in \mathbb{R}$ is a normalization coefficient such that $\alpha_t(u(\Phi_t^{-1}(\cdot)) + t v (\Phi_t^{-1}(\cdot))) \in \mathcal{N}(\Phi_t(\Omega))$, and $v$ is a \textit{corrector} whose choice is unrestricted. 

Note that the corrector $v$ is reminiscent of the concept of \textit{material derivative} $\dot{u}$ of $u$ \cite{sokolowski} which appears in exact formulas for the second-order domain derivative of various functionals whenever such derivative exists (see, e.g., by no means complete list of works \cite{bandle,bouchitte,garabedian,grinfeld,novruzipierre,sokolowski}). 
Roughly speaking, if $\{U_t\}$ is a sufficiently regular with respect to $t$ family of critical points of such a functional, then the material derivative $\dot{U}_0$ can be defined by
$$
U_t(\Phi_t(x)) = U_0(x) + t \dot{U}_0(x) + o(t),
\quad 
x \in \Omega.
$$
In fact, $\dot{U}_0$ can be seen as an \textit{optimal} corrector.
However, in order to use the exact formulas for the second-order domain derivative in particular applications, one is forced to solve a boundary value problem to determine the material derivative of $u$, which is usually a nontrivial task by itself (see Section \ref{sec:optimal} below for a more detailed discussion of this issue). 
The main idea pursued in this paper is that one does not need to find an optimal corrector if he can guess its good approximation based on a physical or geometric intuition. 
A good approximation of the optimal corrector would yield a good upper bound for the second-order finite difference of $\mu(\Phi_t(\Omega))$. 
In particular, if this upper bound is negative and the right-hand side of \eqref{eq:had0} equals zero, then $\mu(\Phi_t(\Omega)) < \mu(\Omega)$ for sufficiently small $t$, which implies the nonoptimality of $\Omega$.

Let us mention that the results of our paper are also applied to sublinear problems of the type \eqref{Dm}, as well as to problems with convex-concave nonlinearities \cite{abc} (for a suitable range of a parameter), since such problems possess least energy solutions. 
Moreover, apart from problems of the type \eqref{Dm}, we obtain in the same way a second-order estimate for the the first eigenvalue of the $p$-Laplacian
\begin{equation}\label{eq:nu0}
\lambda_1(\Omega) = 
\min_{w \in \wolp(\Omega) \setminus \{0\}} 
\frac{\int_{\Omega} |D w|^p \,dx}{\int_{\Omega} |w|^p  \,dx}.
\end{equation}
Note that $\lambda_1(\Phi_t(\Omega))$ is at least once differentiable with respect to $t$, see \cite{garcia, lamberti}.
Moreover, the first eigenfunction of the $p$-Laplacian and least energy solutions of \eqref{Dm} are conceptually the same objects, see \cite{GrumiauParini} for rigorous results in this direction. 

\medskip
The present paper is organized as follows. 
In Section \ref{sec:main}, we state our problem in the full generality and discuss main results. 
In Sections \ref{sec:first-order} and \ref{sec:second_derivative}, we treat the first-order and second-order estimates for $\mu(\Phi_t(\Omega))$, respectively.
Section \ref{sec:optimal} is devoted to the formal discussion of the concept of optimal corrector.
In Section \ref{sec:special}, we consider two particular examples of the main result: the problem \eqref{Dm} and the first eigenvalue of the $p$-Laplacian \eqref{eq:nu0}.
In Section \ref{sec:lower_dim}, we further simplify obtained formulas either in the planar case $N=2$ or under additional assumptions on the perturbation $\Phi_t$. 
Finally, in Section \ref{sec:applications}, we apply our results to study the behaviour of the first eigenvalue of the Laplacian in rectangles under specific perturbations. In some cases, we are able to compare values of the second-order estimate for $\lambda_1(\Omega)$ computed for optimal and nonoptimal correctors.

\section{Main results}\label{sec:main}

Consider a bounded domain $\Omega \subset \mathbb{R}^N$ with the boundary $\partial \Omega$, where $N \geq 2$. Let $R$ and $\widetilde R$ be smooth vector fields over $\overline\Omega$. 
Define a deformed domain $\Omega_t$ as $\Omega_t = \Phi_t(\Omega)$, where
$$
\Phi_t(x) := x + t R(x) + \frac{1}{2} t^2 \widetilde R(x),
\quad 
|t| < \delta,
$$ 
and $\delta>0$ is sufficiently small.

We will work with a general energy functional $E_t$ defined by 
\begin{equation*}
E_t[w] = \int_{\Omega_t} L(w, Dw) \,dy \equiv \int_{\Omega} L(w(\Phi_t(x)), D w(\Phi_t(x)) \cdot \Psi_t(x)) \varphi_t(x) \,dx,
\end{equation*}
where $D$ denotes the corresponding Jacobi matrix, 
\begin{equation}\label{eq:m_and_Psi}
\varphi_t(x) = {\rm det} (D \Phi_t(x))
\quad
\text{and}
\quad
\Psi_t(x) 
= (D(\Phi_t^{-1}(y)))|_{y = \Phi_t(x)}.
\end{equation}
We always assume that $E_t$ obeys the following set of assumptions:
\begin{enumerate}[label={\rm(\roman*)}]
	\item\label{as:2} 
	$E_t \in C^2(\wolp(\Omega_t), \mathbb{R})$ for some $p>1$, and $E_{(\cdot)}[v(\Phi_{(\cdot)}^{-1}(y))] \in C^2((-\delta,\delta),\mathbb{R})$ for any $v \in \wolp(\Omega)$.
	\item\label{as:3} $E_t$ possesses a nonzero least energy critical point, and any such critical point is a minimizer of $E_t$ over the Nehari manifold
	\begin{equation*}
	\mathcal{N}(\Omega_t)
	= 
	\left\{ 
	w \in \wolp(\Omega_t)\setminus \{0\}:~ 
	E_t'[w]w = \int_{\Omega_t} \left((D_z L(w, D w), D w) + L_u(w, D w) w \right) dx = 0
	\right\}.
	\end{equation*}
	We will denote the corresponding least energy critical level as $\mu(\Omega_t)$, that is,
	\begin{equation}\label{eq:mut}
	\mu(\Omega_t) := \inf_{w \in \mathcal{N}(\Omega_t)} E_t[w].
	\end{equation}
	\item\label{as:3.5} For any $w \in \wolp(\Omega)\setminus\{0\}$ there exits $\alpha > 0$ such that $\alpha w \in \mathcal{N}(\Omega)$. 
	Moreover, if $w \in \mathcal{N}(\Omega)$ and $v \in \wolp(\Omega)$, then there exist $\alpha_{(\cdot)}, \beta_{(\cdot)} \in C^1((-\delta,\delta), (0,+\infty))$ such that $\alpha_0=1$,  $\alpha_t \left(w(\Phi_t^{-1}(\cdot)) + t v(\Phi_t^{-1}(\cdot))\right)\in \mathcal{N}(\Omega_t)$, and $\beta_0=1$, $\beta_s \left(w + s v\right)\in \mathcal{N}(\Omega)$.
	\item\label{as:4} There exists a nonzero least energy critical point $u$ of $E_0$ such that 
	$E_0''[u](u,u) \neq 0$. 
\end{enumerate}
\begin{remark}
In Section \ref{sec:special} below, we show that $E_t$ with $L(s,z) = \frac{1}{p}|z|^p - \frac{1}{q}|s|^{q}$, where $p,q \geq 2$, $q < p^*$, and $q \neq p$, satisfies the required assumptions \ref{as:2}-\ref{as:4}.  
Such example of $E_t$ corresponds to the problem \eqref{Dm} and can be kept in mind as the main model case.
\end{remark}

Hereinafter, we will always denote by $u \in \wolp(\Omega)$ a least energy critical point of $E_0$ which satisfies the assumption \ref{as:4}.
Consider a family of functions 
\begin{equation}\label{eq:u_t}
u_t(x) = \alpha_t (u(x) + t v(x)),
\quad 
x \in \Omega,
\end{equation}
and its transposition to $\Omega_t$ defined by
\begin{equation*}
U_t(y) := u_t(\Phi_t^{-1}(y)) 
= 
\alpha_t \left(u(\Phi_t^{-1}(y)) + t v (\Phi_t^{-1}(y))\right),
\quad 
y \in \Omega_t,
\end{equation*}
Here, $v \in \wolp(\Omega)$ is an arbitrary function called a \textit{corrector}, and $\alpha_t > 0$ is chosen in such a way that $U_t \in \mathcal{N}(\Omega_t)$ (see the assumption \ref{as:3.5}).	
Thus, the family $\{U_t\}$ can be called a \textit{Nehari manifold trajectory} emanating from $u$.
After the change of variables, the Nehari manifold constraint on $U_t$ reads as
\begin{equation}\label{Nehari_constraint}
0 = E_t'[U_t]U_t
= \int_{\Omega} (D_z L(u_t, D u_t \cdot \Psi_t),D u_t \cdot \Psi_t) \varphi_t \,dx + \int_{\Omega} L_u(u_t, D u_t \cdot \Psi_t) u_t \varphi_t \,dx.
\end{equation}

Let us denote 
\begin{equation}\label{eq:mt}
m(t) := E_t[U_t] = \int_{\Omega} L(u_t, D u_t \cdot \Psi_t) \varphi_t \,dx.
\end{equation}
By the definition \eqref{eq:mut} of $\mu(\Omega_t)$ and the assumption $U_t \in \mathcal{N}(\Omega_t)$, we have 
$$
\mu(\Omega) = m(0)
\quad
\text{and}
\quad
\mu(\Omega_t) \leq m(t)
\quad \text{for all} \quad
|t| < \delta.
$$ 
Therefore, Taylor's theorem applied to $m(t)$ yields
\begin{equation}\label{eq:taylor}
\mu(\Omega_t) \leq \mu(\Omega) + \dot{m}(0) t + \frac{\ddot{m}(0)}{2} t^2 + o(t^2).
\end{equation}
The first-order derivative $\dot{m}(0)$ in the model case of the Lane-Emden problem \eqref{Dm} with a sufficiently smooth boundary $\partial \Omega$ has the form \eqref{eq:had0}. 
In the general case, we give the following result.
\begin{proposition}\label{prop:first-derivative}
	Let \ref{as:2}-\ref{as:3.5} be satisfied and let $u \in \wolp(\Omega)$ be a least energy critical point of $E_0$. 
	Then 
	\begin{equation}\label{eq:dotm01}
	\dot m(0) 
	=
	\int_{\Omega} L(u, D u) \, {\rm div} (R) \,dx - \int_{\Omega} (D_z L(u, D u ),D u \cdot D R) \,dx.
	\end{equation}
\end{proposition}
\begin{remark}\label{rem:0}
	Let $\Omega$ be of class $C^1$. If either $u \in C^1(\overline{\Omega}) \cap C^2(\Omega)$, or $u \in C^1(\overline{\Omega})$ and $L(s,\cdot)$ is strictly convex for any $s \in \mathbb{R}$, then the integrals in \eqref{eq:dotm01} can be expressed via the Pohozaev identity (see \cite[Theorems 1 and 2]{degiovanni}) as integrals over the boundary of $\Omega$:
	\begin{equation}\label{eq:m0poh}
	\dot m(0) 
	=
	\int_{\partial \Omega} L(0,Du) (R, n) \, d\sigma
	-
	\int_{\partial \Omega} (D_z L(0,Du), Du) (R, n) \, d\sigma,
	\end{equation}
	where $n$ is the outward unit normal vector to $\partial \Omega$.
\end{remark}

The main aim of the present paper is to obtain a formula for the second-order derivative $\ddot{m}(0)$, which allows to estimate $\mu(\Omega_t)$ from above for sufficiently small $|t|$ provided $\dot{m}(0)=0$.

Let us introduce the symmetric bilinear form associated with the second-order variation of $E_0$:
\begin{align}
\notag
\left< g, h \right>_0 
:= 
E_0''[u](g,h) 
&= 
\int_{\Omega} (D^2_{zz}L(u, D u), D g \otimes D h) \,dx 
+ 
\int_{\Omega} (D_z L_u(u, D u), D g) h \,dx 
\\
\label{eq:<>0}
&+ 
\int_{\Omega} (D_z L_u(u, D u),D h) g \,dx 
+ 
\int_{\Omega} L_{uu}(u, D u) g h \,dx,
\quad 
g, h \in \wolp(\Omega),
\end{align}
and the functional $Q$ defined by
\begin{align}
\notag
Q[h] 
= 
&-\int_{\Omega} (D^2_{zz} L(u, D u), D h  \otimes D u \cdot DR) \,dx 
- \int_{\Omega} (D_z L_u(u, D u), D u \cdot DR) h \,dx 
\\
\notag
&- \int_{\Omega} (D_z L(u, D u),D h \cdot DR) \,dx 
+ \int_{\Omega} (D_z L(u, D u), D h) \, {\rm div}(R) \,dx 
\\
\label{eq:Qh}
&+ \int_{\Omega} L_u(u, D u) h \, {\rm div}(R) \,dx,
\quad  h \in \wolp(\Omega).
\end{align}

Our main result is the following theorem.
\begin{thm}\label{thm:main0}
	Let \ref{as:2}-\ref{as:3.5} be satisfied and let $u \in \wolp(\Omega)$ be a least energy critical point of $E_0$ satisfying \ref{as:4}. 
	Then for any corrector $v \in \wolp(\Omega)$ there holds
	\begin{align}
	\notag
	\ddot m(0) 
	&=
	\int_{\Omega} L(u, D u ) \, {\rm div} (\widetilde{R}) \,dx
	-
	\int_{\Omega} (D_z L(u, D u ),D u \cdot D \widetilde{R}) \,dx
	\\
	\notag
	&-2 \int_{\Omega} (D_z L(u, D u ), D u \cdot DR) \, {\rm div}(R) \,dx 
	+ 2 \int_{\Omega} (D_z L(u, D u ),D u \cdot DR \cdot DR) \,dx
	\\
	\notag
	&
	+
	2 \int_{\Omega} L(u, D u ) \chi_2(DR) \,dx
	+ \int_{\Omega} (D^2_{zz} L(u, D u ),D u \cdot DR \otimes D u \cdot DR) \,dx 
	\\
	\label{ddm_intermediate_main0}
	&
	- 
	\frac {(Q[u])^2}{\left< u,u \right>_0} 
	+ 
	2 Q\left[ v - \frac {\left< u,v \right>_0}{\left< u,u \right>_0} u  \right] 
	+
	\left< v - \frac{\left< u, v \right>_0}{\left< u,u \right>_0} u, v - \frac{\left< u,v \right>_0}{\left< u,u \right>_0} u \right>_0,
	\end{align}
	where 
	\begin{equation}\label{eq:chi2}
	\chi_2(DR) = \sum_{i<j} \left( (DR)_{ii} (DR)_{jj} - (DR)_{ij} (DR)_{ji} \right).
	\end{equation}
\end{thm}

Let us now obtain two simplifications of \eqref{ddm_intermediate_main0} under additional assumptions on the corrector.
Notice that $\ddot m(0)$ depends on $v - \frac {\left< u, v \right>_0}{\left< u,u \right>_0} u$, i.e., $v$ minus its projection onto $u$. 
Thus, in order to simplify \eqref{ddm_intermediate_main0}, it is natural to require the orthogonality of $u$ and $v$ in the sense that 
\begin{equation}\label{eq:uv=0}
\left< u, v \right>_0 \equiv E_0''[u](u,v) = 0.
\end{equation}
Under this assumption, we arrive at the following result.
\begin{proposition}\label{lem:second-derivative-uv=0}
	Let \ref{as:2}-\ref{as:3.5} be satisfied and let $u \in \wolp(\Omega)$ be a least energy critical point of $E_0$ satisfying \ref{as:4}.  
	If a corrector $v \in \wolp(\Omega)$ is such that  $\left< u, v \right>_0 = 0$, then
	\begin{align}
	\notag
	\ddot m(0) 
	&=
	\int_{\Omega} L(u, D u ) \,{\rm div} (\widetilde{R}) \,dx
	-
	\int_{\Omega} (D_z L(u, D u ),D u \cdot D \widetilde{R}) \,dx
	\\
	\notag
	&-2 \int_{\Omega} (D_z L(u, D u ), D u \cdot DR) \, {\rm div}(R) \,dx 
	+ 2 \int_{\Omega} (D_z L(u, D u ),D u \cdot DR \cdot DR) \,dx
	\\
	\notag
	&
	+
	2 \int_{\Omega} L(u, D u ) \chi_2(DR) \,dx
	+ \int_{\Omega} (D^2_{zz} L(u, D u ),D u \cdot DR \otimes D u \cdot DR) \,dx 
	\\
	\label{ddm_unoptimized0}
	&- 
	\frac {(Q[u])^2}{\left< u,u \right>_0} 
	+ 
	2 Q[v]
	+  
	\left< v,v \right>_0.
	\end{align}
\end{proposition}

Note that \eqref{ddm_unoptimized0} is, in general, quadratic with respect to $v$. 
As $\ddot m(0)$ serves as an upper estimate on the second superdifferential of $\mu(\Omega_t)$ at $t=0$, it is feasible to obtain a closed-form optimization of \eqref{ddm_unoptimized0} over a class of correctors which is the one-dimensional linear space spanned by $v$.
In Lemma \ref{lem:vv>0} below we show that the orthogonality assumption \eqref{eq:uv=0} implies $\left< v, v \right>_0 \geq 0$. 
In the case of the strict inequality, we get the following result.
\begin{proposition}\label{thm:main}
	Let \ref{as:2}-\ref{as:3.5} be satisfied and let $u \in \wolp(\Omega)$ be a least energy critical point of $E_0$ satisfying \ref{as:4}. 
	If a corrector $v \in \wolp(\Omega)$ is such that $\left< u, v \right>_0 = 0$ and $\left< v, v \right>_0 = -Q[v] > 0$, then
	\begin{align}
	\notag
	\ddot m(0) 
	&=
	\int_{\Omega} L(u, D u ) \, {\rm div} (\widetilde{R}) \,dx
	-
	\int_{\Omega} (D_z L(u, D u ),D u \cdot D \widetilde{R}) \,dx
	\\
	\notag
	&-2 \int_{\Omega} (D_z L(u, D u ), D u \cdot DR) \, {\rm div}(R) \,dx 
	+ 2 \int_{\Omega} (D_z L(u, D u ),D u \cdot DR \cdot DR) \,dx
	\\
	\notag
	&
	+
	2 \int_{\Omega} L(u, D u ) \chi_2(DR) \,dx
	+ \int_{\Omega} (D^2_{zz} L(u, D u ),D u \cdot DR \otimes D u \cdot DR) \,dx 
	\\
	\label{ddm_intermediate_three}
	&- \frac {(Q[u])^2}{\left< u,u \right>_0} 
	-\frac { (Q[v])^2 } {\left< v,v \right>_0}.
	\end{align}
	Moreover, \eqref{ddm_intermediate_three} is optimal on the class of correctors $\{\gamma v: \gamma \in \mathbb{R}\}$ in the sense that if we define 
	\begin{equation}\label{eq:w_t}
	w_t(x) = \beta_t (u(x) + t \gamma v(x)),
	\quad 
	x \in \Omega,
	\end{equation}
	and its transposition to $\Omega_t$ defined by $W_t(y) := w_t(\Phi_t^{-1}(y))$, where $\beta_t > 0$ is chosen such that $W_t \in \mathcal{N}(\Omega_t)$, then for any $\gamma \in \mathbb{R}$,
	\begin{equation*}
	\ddot m(0) \leqslant \left . \frac {\partial^2} {\partial t^2} E[W_t] \right |_{t=0}.
	\end{equation*}
\end{proposition}

\begin{remark}\label{rem:1}
	Under the same regularity assumptions as in Remark \ref{rem:0}, the first two integrals in  \eqref{ddm_intermediate_three} containing $\widetilde{R}$ can be expressed via the Pohozaev identity as boundary integrals in \eqref{eq:m0poh}.
	Moreover, under similar assumptions, according to the structural theorem obtained in \cite{novruzipierre}, it is natural to expect that other integrals in \eqref{ddm_intermediate_three} can be also expressed as integrals over the boundary $\partial \Omega$. We do not provide additional details in the present paper and postpone the corresponding investigations for future research.
\end{remark}

In Section \ref{sec:second_derivative} below, we discuss some additional simplifications of the formula \eqref{ddm_intermediate_main0}.

\section{Auxiliary expressions and first-order derivative}\label{sec:first-order}
In this section, we prove the formula \eqref{eq:dotm01} for $\dot{m}(0)$ stated in Proposition \ref{prop:first-derivative}.
First, let us give expressions for $\dot{\varphi}_0$, $\ddot{\varphi}_0$, $\dot{\Psi}_0$, and $\ddot{\Psi}_0$, as they will be used in the sequel.
Recalling the notations \eqref{eq:m_and_Psi}, we have
$\varphi_t = {\rm det} (I +t DR + \frac 1 2 t^2 D\widetilde{R})$. 
Therefore, $\varphi_0 = 1$ and
\begin{align}
\label{eq:dotphi}
\dot \varphi_0  &= \sum_{i=1}^N (DR)_{ii} = {\rm div} (R);\\
\label{eq:ddotphi}
\ddot \varphi_0 &= \sum_{i,j=1}^N \left( (DR)_{ii} (DR)_{jj} - (DR)_{ij} (DR)_{ji} \right) + \sum_{i=1}^N (D \widetilde{R})_{ii} = 2 \chi_2(DR) + {\rm div} (\widetilde{R}), 
\end{align}
where $\chi_2(DR)$ stands for the third-to-highest coefficient of the characteristic polynomial of the matrix $DR$. That is, if $M$ is a square matrix, then
\begin{equation*}
\chi_2(M) = 
\sum_{i<j} (M_{ii} M_{jj} - M_{ij} M_{ji}).
\end{equation*}

To calculate the derivatives of $\Psi_t$, we use the rules for derivatives of the inverse matrix:
\begin{equation*}
\frac{d}{dt}(M^{-1}) 
=
 - M^{-1} \dot{M} M^{-1}, 
\qquad 
\frac{d^2}{dt^2} (M^{-1}) 
= 2 M^{-1} \dot{M} M^{-1} \dot{M} M^{-1} - M^{-1} \ddot{M} M^{-1}.
\end{equation*}
Thereby, we obtain
\begin{equation}
\label{eq:dotpsi}
\Psi_0 = I, \qquad \dot \Psi_0 = - DR, \qquad \ddot \Psi_0 = 2 DR \cdot DR - D\widetilde{R}.
\end{equation}

\medskip
Let us now deduce the formula  \eqref{eq:dotm01} for $\dot{m}(0)$. 
From the definition \eqref{eq:mt} of $m(t)$, we get
\begin{align}
\notag
\dot m(t) 
&= 
\int_{\Omega} L(u_t, D u_t \cdot \Psi_t) \dot \varphi_t \,dx 
+ 
\int_{\Omega} (D_z L(u_t, D u_t \cdot \Psi_t),D u_t \cdot \dot \Psi_t) \varphi_t \,dx \\
\label{eq:dotmt}
& + 
\int_{\Omega} (D_z L(u_t, D u_t \cdot \Psi_t), D \dot u_t \cdot \Psi_t) \varphi_t \,dx 
+ 
\int_{\Omega} L_u(u_t, D u_t \cdot \Psi_t) \dot u_t \varphi_t \,dx.
\end{align}
Since $\alpha_t$ in the definition \eqref{eq:u_t} of $u_t$ is differentiable (see the assumption \ref{as:3.5}), we have
\begin{equation}\label{dot_u_expansion}
\dot u_t = \dot \alpha_t (u + t v) + \alpha_t v = \frac {\dot \alpha_t}{\alpha_t} u_t + \alpha_t v \in \wolp(\Omega).
\end{equation}
Therefore, recalling that $U_t \in \mathcal{N}(\Omega_t)$, 	
we see that the last two terms in \eqref{eq:dotmt} are, in fact, 
\begin{equation*}
E_t'[U_t](\dot{u}_t(\Phi_t^{-1}(\cdot))) = \frac {\dot \alpha_t} {\alpha_t} E_t'[U_t] U_t + \alpha_t E'_t[U_t] (v(\Phi_t^{-1}(\cdot))) =
\alpha_t E'_t[U_t] (v(\Phi_t^{-1}(\cdot))).
\end{equation*}
Thus, $\dot{m}(t)$ can be rewritten as
\begin{align}
\notag
\dot m(t) 
&= 
\int_{\Omega} L(u_t, D u_t \cdot \Psi_t) \dot \varphi_t \,dx 
+
\int_{\Omega} (D_z L(u_t, D u_t \cdot \Psi_t),D u_t \cdot \dot \Psi_t) \varphi_t \,dx\\
\label{eq:dotmt1}
&+ 
\alpha_t \int_{\Omega} (D_z L(u_t, D u_t \cdot \Psi_t), D v \cdot \Psi_t) \varphi_t \,dx 
+ 
\alpha_t \int_{\Omega} L_u(u_t, D u_t \cdot \Psi_t) v \varphi_t \,dx.
\end{align}
Putting now $t=0$, we obtain
\begin{equation*}
\dot m(0) 
= 
\int_{\Omega} L(u, D u) \dot \varphi_0 \,dx + \int_{\Omega} (D_z L(u, D u ), D u \cdot \dot \Psi_0) \,dx
+
E_0'[u] v.
\end{equation*}
Recalling that $u$ is a critical point of $E_0$, we have $E_0'[u] v = 0$. 
Hence, using the expressions \eqref{eq:dotphi} and \eqref{eq:dotpsi}, we arrive at
\begin{equation*}
\dot m(0) 
=
\int_{\Omega} L(u, D u) \, {\rm div} (R) \,dx - \int_{\Omega} (D_z L(u, D u ),D u \cdot D R) \,dx.
\end{equation*}
Therefore, Proposition \ref{prop:first-derivative} is proved.

\section{Second-order derivative}\label{sec:second_derivative}
In this section, we study the second-order derivative $\ddot{m}(0)$. 
Recall that the expression for $\dot{m}(t)$ is given by \eqref{eq:dotmt1}.
Differentiating \eqref{eq:dotmt1}, we get
\begin{align*}
\ddot m(t) 
&= 
{\int_{\Omega} L(u_t, D u_t \cdot \Psi_t) \ddot \varphi_t \,dx} 
+ 
{2 \int_{\Omega} (D_z L(u_t, D u_t \cdot \Psi_t), D u_t \cdot \dot \Psi_t) \dot \varphi_t \,dx} 
\\
&+ 
\int_{\Omega} (D_z L(u_t, D u_t \cdot \Psi_t),D \dot u_t \cdot \Psi_t) \dot \varphi_t \,dx 
+
\int_{\Omega} L_u(u_t, D u_t \cdot \Psi_t) \dot u_t \dot \varphi_t \,dx 
\\
&+
{\int_{\Omega} (D_z L(u_t, D u_t \cdot \Psi_t),D u_t \cdot \ddot \Psi_t) \varphi_t \,dx}    
+ 
\int_{\Omega} (D_z L(u_t, D u_t \cdot \Psi_t), D \dot u_t \cdot \dot \Psi_t) \varphi_t \,dx 
\\
&+ 
{\int_{\Omega} (D^2_{zz} L(u_t, D u_t \cdot \Psi_t),D u_t \cdot \dot \Psi_t \otimes D u_t \cdot \dot \Psi_t) \varphi_t \,dx} 
\\
&+ 
\int_{\Omega} (D^2_{zz} L(u_t, D u_t \cdot \Psi_t),D \dot u_t \cdot \Psi_t \otimes D u_t \cdot \dot \Psi_t) \varphi_t \,dx 
\\
&+ 
\int_{\Omega} (D_z L_u(u_t, D u_t \cdot \Psi_t),D u_t \cdot \dot \Psi_t) \dot u_t \varphi_t \,dx 
+ 
{\dot \alpha_t E'_t[U_t] (v(\Phi_t^{-1}(\cdot)))} 
\\
&+ 
\alpha_t \int_{\Omega} (D_z L(u_t, D u_t \cdot \Psi_t),D v \cdot \Psi_t) \dot \varphi_t \,dx 
+ 
\alpha_t \int_{\Omega} (D_z L(u_t, D u_t \cdot \Psi_t),D v \cdot \dot \Psi_t) \varphi_t \,dx 
\\
&+ 
\alpha_t \int_{\Omega} (D^2_{zz} L(u_t, D u_t \cdot \Psi_t), D v \cdot \Psi_t \otimes D u_t \cdot \dot \Psi_t) \varphi_t \,dx 
\\
&+ 
\alpha_t \int_{\Omega} (D^2_{zz} L(u_t, D u_t \cdot \Psi_t), D v \cdot \Psi_t \otimes D \dot u_t \cdot \Psi_t) \varphi_t \,dx
\\
&+ 
\alpha_t \int_{\Omega} (D_z L_u(u_t, D u_t \cdot \Psi_t),D v \cdot \Psi_t) \dot u_t \varphi_t \,dx 
\\
&+ 
\alpha_t \int_{\Omega} L_u(u_t, D u_t \cdot \Psi_t) v \dot \varphi_t \,dx 
+ 
\alpha_t \int_{\Omega} (D_z L_u(u_t, D u_t \cdot \Psi_t),D u_t \cdot \dot \Psi_t) v \varphi_t \,dx
\\
&+ 
\alpha_t \int_{\Omega} (D_z L_u(u_t, D u_t \cdot \Psi_t), D \dot u_t \cdot \Psi_t) v \varphi_t \,dx
+ 
\alpha_t \int_{\Omega} L_{uu}(u_t, D u_t \cdot \Psi_t) v \dot u_t \varphi_t \,dx.
\end{align*}
First, each term in the sum
\begin{align}\label{eq:dashed1}
&\int_{\Omega} L(u_t, D u_t \cdot \Psi_t) \ddot \varphi_t \,dx
+ 
2 \int_{\Omega} (D_z L(u_t, D u_t \cdot \Psi_t), D u_t \cdot \dot \Psi_t) \dot \varphi_t \,dx
\\
\notag
&+
\int_{\Omega} (D_z L(u_t, D u_t \cdot \Psi_t),D u_t \cdot \ddot \Psi_t) \varphi_t \,dx
+
\int_{\Omega} (D^2_{zz} L(u_t, D u_t \cdot \Psi_t),D u_t \cdot \dot \Psi_t \otimes D u_t \cdot \dot \Psi_t) \varphi_t \,dx
\end{align}
is of deformation-deformation type, i.e., the differentiation with respect to $t$ appears two times in factors dealing with the deformation. 
Second, the term ${\dot \alpha_t E'_t[U_t] (v(\Phi_t^{-1}(\cdot)))}$ will vanish at $t=0$. 
Third, each term in the sum
\begin{align}
\notag
&\alpha_t \int_{\Omega} (D^2_{zz} L(u_t, D u_t \cdot \Psi_t), D v \cdot \Psi_t \otimes D \dot u_t \cdot \Psi_t) \varphi_t \,dx 
+ 
\alpha_t \int_{\Omega} (D_z L_u(u_t, D u_t \cdot \Psi_t),D v \cdot \Psi_t) \dot u_t \varphi_t \,dx
\\
\label{eq:wavy1}
&+
\alpha_t \int_{\Omega} (D_z L_u(u_t, D u_t \cdot \Psi_t), D \dot u_t \cdot \Psi_t) v \varphi_t \,dx
+ 
\alpha_t \int_{\Omega} L_{uu}(u_t, D u_t \cdot \Psi_t) v \dot u_t \varphi_t \,dx
\end{align}
is of corrector-corrector type, i.e., it contains both $\dot u_t$ and $v$. 
The sum \eqref{eq:wavy1} is transformed to the second-order variation of $E_t$. 
That is, let us introduce the symmetric bilinear form (cf.\ \eqref{eq:<>0})
\begin{align*}
\left< g, h \right>_t 
&= 
E_t''[U_t](g(\Phi_t^{-1}(\cdot)),h(\Phi_t^{-1}(\cdot))) 
\\
&= 
\int_{\Omega} (D^2_{zz}L(u_t, D u_t \cdot \Psi_t), D g \cdot \Psi_t \otimes D h \cdot \Psi_t) \varphi_t \,dx 
\\
&+ 
\int_{\Omega} (D_z L_u(u_t, D u_t \cdot \Psi_t), D g \cdot \Psi_t) h \varphi_t \,dx 
\\
&+ 
\int_{\Omega} (D_z L_u(u_t, D u_t \cdot \Psi_t),D h \cdot \Psi_t) g \varphi_t \,dx 
+ 
\int_{\Omega} L_{uu}(u_t, D u_t \cdot \Psi_t) g h \varphi_t \,dx,
\quad
g,h \in \wolp(\Omega).
\end{align*}
Then, using \eqref{dot_u_expansion}, the sum \eqref{eq:wavy1} can be compressed as
\begin{equation}\label{eq:wavy2}
\alpha_t \left< v, \dot u_t \right>_t = \alpha_t \left< v,v \right>_t + \dot \alpha_t \left< v, u_t \right>_t.
\end{equation}

To catch the structure of $\ddot{m}(t)$, let us regroup the expression for $\ddot{m}(t)$ in the following way:

\hspace*{-1.2cm}
\vbox{\begin{align*}
\ddot m(t) 
= 
{\int_{\Omega} L(u_t, D u_t \cdot \Psi_t) \ddot \varphi_t \,dx} 
&+ 
{2 \int_{\Omega} (D_z L(u_t, D u_t \cdot \Psi_t), D u_t \cdot \dot \Psi_t) \dot \varphi_t \,dx} 
\\
+ 
{\int_{\Omega} (D_z L(u_t, D u_t \cdot \Psi_t),D u_t \cdot \ddot \Psi_t) \varphi_t \,dx} 
&+ 
{\int_{\Omega} (D^2_{zz} L(u_t, D u_t \cdot \Psi_t),D u_t \cdot \dot \Psi_t \otimes D u_t \cdot \dot \Psi_t) \varphi_t \,dx} 
\\
+ 
\int_{\Omega} (D^2_{zz} L(u_t, D u_t \cdot \Psi_t),D \dot u_t \cdot \Psi_t \otimes D u_t \cdot \dot \Psi_t) \varphi_t \,dx 
&+ 
\alpha_t \int_{\Omega} (D^2_{zz} L(u_t, D u_t \cdot \Psi_t), D v \cdot \Psi_t \otimes D u_t \cdot \dot \Psi_t) \varphi_t \,dx 
\\
+ 
\int_{\Omega} (D_z L_u(u_t, D u_t \cdot \Psi_t),D u_t \cdot \dot \Psi_t) \dot u_t \varphi_t \,dx &+ 
\alpha_t \int_{\Omega} (D_z L_u(u_t, D u_t \cdot \Psi_t),D u_t \cdot \dot \Psi_t) v \varphi_t \,dx 
\\
+ 
\int_{\Omega} (D_z L(u_t, D u_t \cdot \Psi_t), D \dot u_t \cdot \dot \Psi_t) \varphi_t \,dx 
&+ 
\alpha_t \int_{\Omega} (D_z L(u_t, D u_t \cdot \Psi_t),D v \cdot \dot \Psi_t) \varphi_t \,dx 
\\
+ 
\int_{\Omega} (D_z L(u_t, D u_t \cdot \Psi_t),D \dot u_t \cdot \Psi_t) \dot \varphi_t \,dx 
&+ 
\alpha_t \int_{\Omega} (D_z L(u_t, D u_t \cdot \Psi_t),D v \cdot \Psi_t) \dot \varphi_t \,dx 
\\
+ 
\int_{\Omega} L_u(u_t, D u_t \cdot \Psi_t) \dot u_t \dot \varphi_t \,dx 
&+ 
\alpha_t \int_{\Omega} L_u(u_t, D u_t \cdot \Psi_t) v \dot \varphi_t \,dx 
\\
&+ 
{{\dot \alpha_t E'_t[U_t] (v(\Phi_t^{-1}(\cdot)))}}
+ 
\alpha_t \left< v,v \right>_t + \dot \alpha_t \left< v, u_t \right>_t.
\end{align*}}

\noindent
Putting $t=0$ and noting that $\dot{u} = \dot{\alpha}_0 u + v$ and $E_0'[u]v=0$, we obtain
\begin{align*}
\ddot m(0) 
= 
{\int_{\Omega} L(u, D u) \ddot \varphi_0 \,dx} 
&+ 
{2 \int_{\Omega} (D_z L(u, D u), D u \cdot \dot \Psi_0) \dot \varphi_0 \,dx} 
\\+ 
{\int_{\Omega} (D_z L(u, D u),D u \cdot \ddot \Psi_0) \,dx} 
&+ 
{\int_{\Omega} (D^2_{zz} L(u, D u),D u \cdot \dot \Psi_0 \otimes D u \cdot \dot \Psi_0) \,dx} 
\\
+ 
{\dot\alpha_0} \int_{\Omega} (D^2_{zz} L(u, D u),D u  \otimes D u \cdot \dot \Psi_0) \,dx 
&+ 
2  \int_{\Omega} (D^2_{zz} L(u, D u), D v  \otimes D u \cdot \dot \Psi_0) \,dx 
\\
+  
{\dot\alpha_0} \int_{\Omega} (D_z L_u(u, D u),D u \cdot \dot \Psi_0) u \,dx 
&+ 
2  \int_{\Omega} (D_z L_u(u, D u),D u \cdot \dot \Psi_0) v \,dx
\\
+  
{\dot\alpha_0} \int_{\Omega} (D_z L(u, D u), D u \cdot \dot \Psi_0) \,dx 
&+ 
2  \int_{\Omega} (D_z L(u, D u),D v \cdot \dot \Psi_0) \,dx 
\\
+  
{\dot\alpha_0} \int_{\Omega} (D_z L(u, D u),D u) \dot \varphi_0 \,dx 
&+ 
2  \int_{\Omega} (D_z L(u, D u),D v ) \dot \varphi_0 \,dx 
\\
+  
{\dot\alpha_0} \int_{\Omega} L_u(u, D u) u \dot \varphi_0 \,dx 
&+ 
2  \int_{\Omega} L_u(u, D u) v \dot \varphi_0 \,dx
\\
&+  
\left< v,v \right>_0 + \dot\alpha_0 \left< v, u \right>_0.
\end{align*}
Let us define a linear functional $Q: \wolp(\Omega) \mapsto \mathbb{R}$ as
\begin{align*}
Q[h] 
&= 
\int_{\Omega} (D^2_{zz} L(u, D u), D h  \otimes D u \cdot \dot \Psi_0) \,dx 
+ 
\int_{\Omega} (D_z L_u(u, D u),D u \cdot \dot \Psi_0) h \,dx 
\\
&+ 
\int_{\Omega} (D_z L(u, D u),D h \cdot \dot \Psi_0) \,dx 
+ 
\int_{\Omega} (D_z L(u, D u),D h ) \dot \varphi_0 \,dx 
+ 
\int_{\Omega} L_u(u, D u) h \dot \varphi_0 \,dx
\end{align*}
(see also the equivalent definition \eqref{eq:Qh} of $Q$ written via \eqref{eq:dotphi} and \eqref{eq:dotpsi}).
That is, $Q$ collects all the terms in $\ddot{m}(0)$ except \eqref{eq:dashed1} and \eqref{eq:wavy2} calculated at $t=0$. Such terms come out of differentiating $\dot{m}(t)$ when the derivative falls ones on the deformation coefficient $\Psi_t$ or $\varphi_t$.
Then, $\ddot{m}(0)$ can be compactly written as follows:
\begin{align}
\notag
\ddot m(0) 
&= 
{\int_{\Omega} L(u, D u) \ddot \varphi_0 \,dx} 
+ 
{2 \int_{\Omega} (D_z L(u, D u), D u \cdot \dot \Psi_0) \dot \varphi_0 \,dx} 
\\
\notag
&+ 
{\int_{\Omega} (D_z L(u, D u),D u \cdot \ddot \Psi_0) \,dx} 
+ 
{\int_{\Omega} (D^2_{zz} L(u, D u),D u \cdot \dot \Psi_0 \otimes D u \cdot \dot \Psi_0) \,dx}  
\\
\label{ddm_intermediate}
&+  
{\dot\alpha_0} Q[u]
+ 
2 Q[v]
+  
\left< v,v \right>_0 
+ 
\dot\alpha_0 \left< v, u \right>_0.
\end{align}

\medskip
Let us now find the expression for $\dot \alpha_0$. To this end, we differentiate the constraint $E_t'[U_t]U_t = 0$ given by \eqref{Nehari_constraint}: 
\begin{align*}
0 
&= 
\int_{\Omega} (D_z L(u_t, D u_t \cdot \Psi_t), D u_t \cdot \Psi_t) \dot \varphi_t \,dx
+
\int_{\Omega} (D_z L(u_t, D u_t \cdot \Psi_t), D u_t \cdot \dot \Psi_t) \varphi_t \,dx
\\
&+
\int_{\Omega} (D_z L(u_t, D u_t \cdot \Psi_t), D \dot u_t \cdot \Psi_t) \varphi_t \,dx
+ \int_{\Omega} (D^2_{zz} L(u_t, D u_t \cdot \Psi_t), D u_t \cdot \dot \Psi_t \otimes D u_t \cdot \Psi_t) \varphi_t \,dx
\\
&
+
\int_{\Omega} (D^2_{zz} L(u_t, D u_t \cdot \Psi_t), D \dot u_t \cdot \Psi_t \otimes D u_t \cdot \Psi_t) \varphi_t \,dx
+ \int_{\Omega} (D_z L_u(u_t, D u_t \cdot \Psi_t), D u_t \cdot \Psi_t) \dot u_t \varphi_t \,dx
 \\
 &+ 
 \int_{\Omega} L_u(u_t, D u_t \cdot \Psi_t) u_t \dot \varphi_t \,dx
 +
 \int_{\Omega} L_u(u_t, D u_t \cdot \Psi_t) \dot u_t \varphi_t \,dx
 \\
 &+ 
 \int_{\Omega} (D_z L_u(u_t, D u_t \cdot \Psi_t), D u_t \cdot \dot \Psi_t) u_t \varphi_t \,dx
 +
 \int_{\Omega} (D_z L_u(u_t, D u_t \cdot \Psi_t), D \dot u_t \cdot \Psi_t) u_t \varphi_t \,dx
 \\
 &+ \int_{\Omega} L_{uu}(u_t, D u_t \cdot \Psi_t) \dot u_t u_t \varphi_t \,dx.
\end{align*}
Putting $t=0$, we obtain
\begin{align}
\notag
0 
&= 
{\int_{\Omega} (D_z L(u, D u), D u) \dot \varphi_0 \,dx}
+
{\int_{\Omega} (D_z L(u, D u), D u \cdot \dot \Psi_0) \,dx}
\\
\notag
&+
\int_{\Omega} (D_z L(u, D u), D \dot u) \,dx
+ 
{\int_{\Omega} (D^2_{zz} L(u, D u), D u \cdot \dot \Psi_0 \otimes D u) \,dx}
\\
\notag
&
+
{\int_{\Omega} (D^2_{zz} L(u, D u), D \dot u \otimes D u) \,dx}
+ 
{\int_{\Omega} (D_z L_u(u, D u), D u) \dot u \,dx}
\\
\notag
&+ 
{\int_{\Omega} L_u(u, D u) u \dot \varphi_0 \,dx}
+
\int_{\Omega} L_u(u, D u) \dot u \,dx
\\
\notag
&+ 
{\int_{\Omega} (D_z L_u(u, D u), D u \cdot \dot \Psi_0) u \,dx}
+
{\int_{\Omega} (D_z L_u(u, D u), D \dot u) u \,dx}
\\
\label{eq:neh_0}
&+ 
{\int_{\Omega} L_{uu}(u, D u) \dot u u \,dx}.
\end{align}
Notice that
\begin{align*}
\left<\dot{u}, u\right>_0
\equiv
E_0''[u](\dot u, u)
&=
{\int_{\Omega} (D^2_{zz} L(u, D u), D \dot u \otimes D u) \,dx}
+ 
{\int_{\Omega} (D_z L_u(u, D u), D u) \dot u \,dx}
\\
&+
{\int_{\Omega} (D_z L_u(u, D u), D \dot u) u \,dx}
+ 
{\int_{\Omega} L_{uu}(u, D u) \dot u u \,dx},
\end{align*}
and
$$
E_0'[u] \dot u = \int_{\Omega} (D_z L(u, D u), D \dot u) \,dx + \int_{\Omega} L_u(u, D u) \dot u \,dx = 0,
$$
and the remaining terms in \eqref{eq:neh_0} are $Q[u]$. 
Therefore, using \eqref{dot_u_expansion}, we compress \eqref{eq:neh_0} as follows:
\begin{equation}\label{eq:alphadot}
0 
= 
\left<\dot u, u\right>_0 + Q[u] 
= 
\dot\alpha_0  \left< u,u \right>_0 + \left< u,v \right>_0 + Q[u],
~~\text{ thus }~~
\dot \alpha_0 =  - \frac {\left< u,v \right>_0}{\left< u,u \right>_0} - \frac {Q[u]}{\left< u,u \right>_0},
\end{equation}
where $\left< u,u \right>_0 \neq 0$ by the assumption \ref{as:4}.

Employing now \eqref{eq:alphadot}, we rewrite the last four terms in \eqref{ddm_intermediate} in the following way:
\begin{align*}
{\dot\alpha_0} Q[u]
+ 
2 Q[v] 
&+ 
\left< v,v \right>_0 
+  
\dot\alpha_0 \left< v, u \right>_0
\\
&= 
\left< v,v \right>_0 
- 
\frac {\left< u,v \right>_0^2}{\left< u,u \right>_0} 
- 
\frac {\left< u,v \right>_0}{\left< u,u \right>_0} Q[u] 
- 
\frac {\left< u,v \right>_0}{\left< u,u \right>_0} Q[u]
- 
\frac {(Q[u])^2}{\left< u,u \right>_0} + 2 Q[v]
\\
&= 
- 
\frac {(Q[u])^2}{\left< u,u \right>_0}
+ 
2 Q\left[ v - \frac {\left< u,v \right>_0}{\left< u,u \right>_0} u  \right] 
+
\left< v - \frac{\left< u, v \right>_0}{\left< u,u \right>_0} u, v - \frac{\left< u,v \right>_0}{\left< u,u \right>_0} u \right>_0.
\end{align*}
Hence, $\ddot m(0)$ can be written as
\begin{align}
\notag
\ddot m(0) 
&= 
{\int_{\Omega} L(u, D u) \ddot \varphi_0 \,dx} 
+ 
{2 \int_{\Omega} (D_z L(u, D u), D u \cdot \dot \Psi_0) \dot \varphi_0 \,dx} 
\\
\notag
&+ 
{\int_{\Omega} (D_z L(u, D u),D u \cdot \ddot \Psi_0) \,dx} 
+ 
{\int_{\Omega} (D^2_{zz} L(u, D u),D u \cdot \dot \Psi_0 \otimes D u \cdot \dot \Psi_0) \,dx}  
\\
\label{ddm_intermediate_main}
&
- 
\frac {(Q[u])^2}{\left< u,u \right>_0} 
+ 
2 Q\left[ v - \frac {\left< u,v \right>_0}{\left< u,u \right>_0} u  \right] 
+
\left< v - \frac{\left< u, v \right>_0}{\left< u,u \right>_0} u, v - \frac{\left< u,v \right>_0}{\left< u,u \right>_0} u \right>_0.
\end{align}
Substituting the expressions \eqref{eq:dotphi}, \eqref{eq:ddotphi}, and \eqref{eq:dotpsi} into \eqref{ddm_intermediate_main}, we obtain \eqref{ddm_intermediate_main0}, and hence Theorem \ref{thm:main0} is established.

\medskip
Let us discuss simplifications of \eqref{ddm_intermediate_main} given by Propositions \ref{lem:second-derivative-uv=0} and \ref{thm:main}. 
Clearly, if $\left< u, v\right>_0 = 0$, then \eqref{ddm_intermediate_main} reads as
\begin{align}
\notag
\ddot m(0) 
&= 
{\int_{\Omega} L(u, D u) \ddot \varphi_0 \,dx} 
+ 
{2 \int_{\Omega} (D_z L(u, D u), D u \cdot \dot \Psi_0) \dot \varphi_0 \,dx} 
\\
\notag
&+ 
{\int_{\Omega} (D_z L(u, D u),D u \cdot \ddot \Psi_0) \,dx} 
+ 
{\int_{\Omega} (D^2_{zz} L(u, D u),D u \cdot \dot \Psi_0 \otimes D u \cdot \dot \Psi_0) \,dx}  
\\
\label{ddm_unoptimized}
&- 
\frac {(Q[u])^2}{\left< u,u \right>_0} 
+ 
2 Q[v]
+  
\left< v,v \right>_0,
\end{align}
which is the result of Proposition \ref{lem:second-derivative-uv=0}. 
Moreover, we have the following information on the sign of $\left< v, v \right>_0$.
\begin{lemma}\label{lem:vv>0}
	Let \ref{as:2}-\ref{as:3.5} be satisfied and let $u \in \wolp(\Omega)$ be a nonzero least energy critical point of $E_0$. 
	If a corrector $v \in \wolp(\Omega)$ is such that $\left< u, v \right>_0 = 0$, then $\left< v, v \right>_0 \geq 0$.
\end{lemma}
\begin{proof}
 	Consider the function $\beta_s (u + sv) \in \wolp(\Omega)$, where $|s|<\epsilon$ and $\epsilon>0$ is sufficiently small, and the normalization coefficient $\beta_{(\cdot)} \in C^1((-\epsilon,\epsilon), (0,+\infty))$ is such that $\beta_0=1$ and $\beta_s (u + sv) \in \mathcal{N}(\Omega)$. In view of the assumption \ref{as:3.5}, such $\beta_s$ exists. In particular, we have
 	$$
 	E_0'[\beta_s (u + sv)](u + sv) = 0
 	\quad   	
 	\text{for any }
 	|s| < \epsilon.
 	$$
 	Since $u$ is a global minimizer of $E_0$ over the Nehari manifold $\mathcal{N}(\Omega)$ (see the assumption \ref{as:3}), we have $E_0[u] \leq E_0[\beta_s (u + sv)]$, 
 	$$
 	\left.\frac{\partial}{\partial s} E_0[\beta_s (u + sv)]\right|_{s=0} 
 	= 
 	E_0'[u] (\dot \beta_0 u + v) = 0
 	\quad
 	\text{and}
 	\quad 
 	\left.\frac{\partial^2}{\partial s^2} E_0[\beta_s (u + sv)]\right|_{s=0}  \geq 0.
 	$$
 	Therefore,
 	\begin{align*}
 	0 
 	&\leq \left.\frac{\partial^2}{\partial s^2} E_0[\beta_s (u + sv)]\right|_{s=0} 
 	= 
 	\left.\frac{\partial}{\partial s} \left(
 	E_0'[\beta_s (u + sv)](\dot \beta_s (u + sv) + \beta_s v)
 	\right) \right|_{s=0}\\
 	&=
 	\left.\frac{\partial}{\partial s} \left(
 	E_0'[\beta_s (u + sv)](\beta_s v)
 	\right) \right|_{s=0} =
 	\dot \beta_0 E_0'[u] v 
 	+ 
 	\left. E_0''[\beta_s (u + sv)](v, \dot \beta_s (u + sv) + \beta_s v)\right|_{s=0}
 	\\
 	&=
 	E_0''[u](v, \dot \beta_0 u + v)
 	=
	\dot \beta_0 E_0''[u](v, u) + E_0''[u](v, v) = \dot \beta_0 \left< u, v \right>_0 + \left< v, v \right>_0 
	= \left< v, v \right>_0.
	\pushQED{\qed}
	\qedhere
 	\end{align*}
\end{proof}

\medskip
We see from Lemma \ref{lem:vv>0} that if $\left<u,v\right>_0=0$, then there are two possibilities: either $\left< v, v \right>_0 = 0$ or $\left< v, v \right>_0 > 0$.
Suppose first that $\left< v, v \right>_0 = 0$, i.e., a degeneracy occurs. Then we trivially obtain from \eqref{ddm_unoptimized} that
\begin{align}
\notag
\ddot m(0) 
&= 
\int_{\Omega} L(u, D u ) \ddot \varphi_0 \,dx 
+ 
2 \int_{\Omega} (D_z L(u, D u ), D u \cdot \dot \Psi_0) \dot \varphi_0 \,dx 
\\
\notag
&+ 
\int_{\Omega} (D_z L(u, D u ),D u \cdot \ddot \Psi_0) \,dx 
+ 
\int_{\Omega} (D^2_{zz} L(u, D u ),D u \cdot \dot \Psi_0 \otimes D u \cdot \dot \Psi_0) \,dx 
\\
\label{ddm_unoptimized_2}
&- 
\frac {(Q[u])^2}{\left< u,u \right>_0} 
+ 
2 Q[v].
\end{align}
It is not hard to see that $\left< u, \gamma v \right>_0 = 0$ and $\left< \gamma v, \gamma v \right>_0 = 0$ for any $\gamma \in \mathbb{R}$, and hence \eqref{ddm_unoptimized_2} remains valid after replacing $v$ by $\gamma v$. 
Thus, the map $\gamma \mapsto \ddot m(0)$ is a polynomial of degree at most one.
If $Q[v] \neq 0$, then we can find $\gamma$ with sufficiently large absolute value in order to achieve $\ddot m(0) < 0$. 
That is, we have shown the following result.
\begin{lemma}\label{lem:second-derivative-vv=0}
 	Let \ref{as:2}-\ref{as:3.5} be satisfied and let $u \in \wolp(\Omega)$ be a least energy critical point of $E_0$ satisfying \ref{as:4}. 
 	If a corrector $v \in \wolp(\Omega)$ is such that $\left< u, v \right>_0 = 0$ and $\left< v, v \right>_0 = 0$, then \eqref{ddm_unoptimized_2} holds true. 
 	Moreover, if $Q[v] \neq 0$, then there exists $\gamma \in \mathbb{R}$ such that $\ddot m(0) < 0$ after replacing $v$ by $\gamma v$.
\end{lemma}

Let us suppose now that $\left< v, v \right>_0 > 0$. As above, $\left< u, \gamma v \right>_0 = 0$ for any $\gamma \in \mathbb{R}$, that is, \eqref{ddm_unoptimized} is valid with the corrector $\gamma v$ instead of $v$. 
Therefore, we see from \eqref{ddm_unoptimized} that the map $\gamma \mapsto \ddot m(0)$ is a quadratic polynomial whose major coefficient $\left< v,v \right>_0$ is positive. 
Evidently, for any $a, b \in \mathbb{R}$, $c>0$, there hold
$$
\min_{\gamma \in \mathbb{R}}\{a + 2b + c \gamma^2\} = -\frac{b^2}{c}+a
\quad \text{and} \quad 
\arg\min_{\gamma \in \mathbb{R}}\{a + 2b + c \gamma^2\} = -\frac{b}{c}.
$$
Applying these facts to $\gamma \mapsto \ddot m(0)$, we see that this quadratic polynomial attains a global minimum at $\gamma = \frac{-Q[v]}{\left<v,v\right>_0}$, which is 
\begin{align}
\notag
\ddot m(0) 
&= 
\int_{\Omega} L(u, D u ) \ddot \varphi_0 \,dx 
+ 
2 \int_{\Omega} (D_z L(u, D u ), D u \cdot \dot \Psi_0) \dot \varphi_0 \,dx 
\\
\notag
&+ 
\int_{\Omega} (D_z L(u, D u ),D u \cdot \ddot \Psi_0) \,dx 
+ 
\int_{\Omega} (D^2_{zz} L(u, D u ),D u \cdot \dot \Psi_0 \otimes D u \cdot \dot \Psi_0) \,dx 
\\
\label{ddm_line_optimized}
&- \frac {(Q[u])^2}{\left< u,u \right>_0} 
-\frac { (Q[v])^2 } {  \left< v,v \right>_0}.
\end{align}
Thus, if $\left<v,v\right>_0 = -Q[v]>0$, then the expression of $\ddot{m}(0)$ for the corrector $v$ has the form \eqref{ddm_line_optimized}, and \eqref{ddm_line_optimized} is minimal on the class of correctors $\{\gamma v: \gamma \in \mathbb{R}\}$. 
Substituting the expressions \eqref{eq:dotphi}, \eqref{eq:ddotphi}, and \eqref{eq:dotpsi} into \eqref{ddm_line_optimized}, we obtain \eqref{ddm_intermediate_three}, which establishes Proposition \ref{thm:main}.

\section{Optimal corrector}\label{sec:optimal}
In this section, we discuss in a formal way an optimality of the choice of a corrector. 
We will work with the expression for $\ddot{m}(0)$ given by \eqref{ddm_line_optimized} (or, equivalently, \eqref{ddm_intermediate_three} of Proposition \ref{thm:main}). 
Notice that the sum
\begin{align*}
&\int_{\Omega} L(u, D u ) \ddot \varphi_0 \,dx 
+ 
2 \int_{\Omega} (D_z L(u, D u ), D u \cdot \dot \Psi_0) \dot \varphi_0 \,dx 
\\
\notag
&+ 
\int_{\Omega} (D_z L(u, D u ),D u \cdot \ddot \Psi_0) \,dx 
+ 
\int_{\Omega} (D^2_{zz} L(u, D u ),D u \cdot \dot \Psi_0 \otimes D u \cdot \dot \Psi_0) \,dx 
- \frac {(Q[u])^2}{\left< u,u \right>_0}
\end{align*}
in \eqref{ddm_line_optimized} depends solely on $u$, $R$, $\widetilde{R}$, and does not depend on a corrector $v$. 
Considered alone, this sum is expected to be positive. 
However, the last fraction $-\frac {(Q[v])^2}{\langle v,v \rangle_0}$ in \eqref{ddm_line_optimized} depends on $u$, $R$, $v$, and is nonpositive, which gives a possibility to prove that $\ddot m(0) < 0$. 
This inequality in combination with \eqref{eq:taylor} and the assumption $\dot m(0) = 0$ implies $\mu(\Omega_t) < \mu(\Omega)$ for sufficiently small $|t|$, which in turn means that $\Omega$ is not optimal. 
In that context, it is natural to call a corrector $v$ \textit{optimal} whenever it maximizes $\frac {(Q[v])^2}{\langle v,v \rangle_0}$ and satisfies $\left<v,v\right>_0 = -Q[v]>0$.

\subsection{Boundary value problem for optimal corrector}\label{sec:optimal-bvp}

Let us consider the maximization problem of finding the optimal corrector:
\begin{equation}\label{eq:maximizer1}
\lambda 
= 
\sup
\left\{
\frac{(Q[v])^2}{\left<v,v\right>_0}:~ v \in \wolp(\Omega) \setminus \{0\},~ \left<u,v\right>_0 = 0 
\right\}.
\end{equation}
In view of the homogeneity of the quotient $\frac{(Q[v])^2}{\left<v,v\right>_0}$, if $\lambda>0$ and it possesses a maximizer $v$, then the normalization constraint $\left<v,v\right>_0 = -Q[v]>0$ is achieved by a proper rescaling of $v$, see below. 	  
The following result can be obtained in a standard way via the Lagrange multipliers rule.
\begin{lemma}\label{lem:optimal}
	Assume that $\lambda>0$ and it possesses a maximizer $v$. 
	Then there exists $\beta \in \mathbb{R}$ such that $v$ satisfies 
	\begin{equation}\label{eq:opt1}
	2 Q[v] Q[h] - 2 \lambda \langle v,h \rangle_0 - \beta \langle u,h \rangle_0 = 0 
	~\text{ for all }~  
	h \in \wolp(\Omega).
	\end{equation}
\end{lemma}

In some particular cases it is possible to prove that in \eqref{eq:opt1} either $\beta = 0$ or $\langle u,h \rangle_0$ is identically zero, see, e.g., the case $Q[u]=0$ and the case of eigenvalue problems discussed in Remark \ref{rem:qu=0} and Section \ref{sec:special_eigenvalue} below, respectively.
If this is true, then the equation \eqref{eq:opt1} becomes
\begin{equation*}
Q[v] Q[h] - \lambda \langle v,h \rangle_0 = 0 ~\text{ for all }~ h \in \wolp(\Omega).
\end{equation*}
Taking $h=u$, we obtain $Q[u]=0$, and introducing $w = -\frac{\lambda}{Q[v]} v$, we arrive at
\begin{equation}
\label{eq:optimal_corrector}
\langle w,h \rangle_0 = -Q[h]
~\text{ for all }~ h \in \wolp(\Omega).
\end{equation}
Clearly, any solution of \eqref{eq:optimal_corrector} is also a maximizer of \eqref{eq:maximizer1}. 
Moreover, testing \eqref{eq:optimal_corrector} by $h=w$, we conclude that $\left<w,w\right>_0 = -Q[w]>0$. 
Therefore, under the above-mentioned hypotheses, \eqref{eq:optimal_corrector} can be seen as the boundary value problem for an optimal corrector. 

Let us remark that the problem \eqref{eq:optimal_corrector} is linear with respect to $w$ and its right-hand side depends linearly on $R$. The coefficients of the corresponding linear functionals, however, may depend on $u$ in a nonlinear way, and $u$ is usually not known in a closed form, which leads to difficulties when one tries to solve such a boundary value problem.

\subsection{Relation to minimizing trajectory}
An optimal corrector is closely related to the concept of minimizing trajectories. 
By the assumption \ref{as:3}, the functional $E_t$ possesses a least energy critical point $U_t$ for every $t \in (-\delta,\delta)$. That is, after the change of variables, we have 
\begin{equation}\label{eq:Euler_in_deformed_domains}
\int_{\Omega} \left [ (D_z L(u_t, D u_t \cdot \Psi_t), D h \cdot \Psi_t) + L_u(u_t, D u_t \cdot \Psi_t) h \right ] \varphi_t \,dx = 0 ~\text{ for all }~ h \in \wolp(\Omega),
\end{equation}
where
\begin{equation*}
u_t(x) = U_t(\Phi_t(x)), \quad x \in \Omega.
\end{equation*}
We call the family $\{U_t\}$ a \textit{minimizing trajectory}. 
Let us suppose that the family $\{U_t\}$ is smooth in the sense that $u_t$ is a differentiable function $(-\delta,\delta) \to \wolp(\Omega)$.
In this case, we can differentiate \eqref{eq:Euler_in_deformed_domains} by $t$ and obtain, after setting $t=0$, that
\begin{equation*}
E_0''[u_0](\dot u_0,h) + Q[h] = 0
~\text{ for all }~ h \in \wolp(\Omega).
\end{equation*}
This problem coincides with the problem \eqref{eq:optimal_corrector} and thus yields both the second-order derivative $\ddot{\mu}(\Omega)$ and the optimal corrector $w=\dot{u}_0$.

However, let us emphasize again that the minimizing trajectory $\{U_t\}$ can be very ``degenerate'' if one talks about superlinear problems of the type \eqref{Dm}. 
Namely, both the continuity and differentiability of such family is uncertain. 
Examples of a discontinuous minimizing trajectory can be easily constructed. 
For instance, consider the Lane-Emden problem \eqref{Dm} on a concentric spherical shell $\Omega_t = B_R \setminus \overline{B}_r$ of the width $t = R-r$, where $B_R$, $B_r$ are the balls of radius $R, r$, respectively, centred at the origin. 
The width $t$ can be taken small enough in order to guarantee that any least energy critical point of $E_t$ on $\Omega_t$ is \textit{nonradial} (see \cite{coffman,Kolon,Nazar2} for the existence results). In view of the isotropy of $E_t$ and radial symmetry of $\Omega_t$, every rotation $R U_t$ of a fixed least energy critical point $U_t$ is again a least energy critical point  of $E_t$. 
Therefore, taking for each sufficiently small $t$ an appropriate rotation $R U_t$, we obtain a discontinuous minimizing trajectory $\{RU_t\}$. 
On the other hand, even if we have a continuous family $\{U_t\}$, its differentiability still cannot be guaranteed, because it is usually proven by a variant of the inverse function theorem which requires the quadratic form $E_t''[U_t](h,h)$ to be nondegenerate. 
However, the concentric spherical shell $\Omega_t$ with sufficiently small $t$ again provides a counterexample: there exists a nonzero $\widetilde{h} \in \wolp(\Omega_t)$ such that $E_t''[U_t](\widetilde{h},\widetilde{h}) = 0$, see, e.g., \cite[Proposition 4.2]{CES}.

\section[Special cases with p-Laplacian]{Special cases with $p$-Laplacian}\label{sec:special}
In this section, we discuss a special case of Proposition \ref{thm:main} for the Lane-Emden problem \eqref{Dm} and obtain an analogue of Proposition \ref{thm:main} for the first eigenvalue \eqref{eq:nu0} of the $p$-Laplacian.

\subsection{Lane-Emden problem}\label{sec:special_nonhom}

Let $L(u,z) = \frac{1}{p} |z|^p - \frac{1}{q} |u|^q$, where $p, q \geq 2$, $q \neq p$, and $q < p^*$.
Critical points of the functional
$$
E_t[w] = \int_{\Omega_t} L(w,Dw) \,dx
$$ 
are in one-to-one correspondence with weak solutions of the problem \eqref{Dm} in $\Omega_t$. Note that both superlinear ($p<q$) and sublinear ($p>q$) behaviours are covered.

To be able to apply Proposition \ref{thm:main}, let us show that the assumptions \ref{as:2}-\ref{as:4} of Section \ref{sec:main} are fulfilled. 
It is not hard to see that the assumption \ref{as:2} holds true since $p,q \geq 2$. 
If $p<q$, then the validity of the assumption \ref{as:3} was proved, e.g., in \cite[Section 2]{GrumiauParini} (see also \cite[Theorem 19]{szulkin-weth}). 
If $p>q$, then $E_t$ has a nonzero 	global minimizer over $\wolp(\Omega_t)$ which is a critical point of $E_t$ and hence belongs to $\mathcal{N}(\Omega_t)$. That is, a global minimizer is a least energy critical point and it can be obtained as a minimizer of $\mu(\Omega_t)$. 
The first part of the assumption \ref{as:3.5} is standard, see, e.g., \cite[Lemma A.1]{BK}.
The second part of the assumption \ref{as:3.5} can be established in much the same way as in \cite[Lemma 2.5]{BK}. We provide the corresponding proof for the sake of completeness. 
\begin{lemma}\label{lem:as3}
	If $w \in \mathcal{N}(\Omega)$ and $v \in \wolp(\Omega)$, then there exist $\alpha_{(\cdot)}, \beta_{(\cdot)} \in C^1((-\delta,\delta), (0,+\infty))$ such that $\alpha_0=1$,  $\alpha_t \left(w(\Phi_t^{-1}(\cdot)) + t v(\Phi_t^{-1}(\cdot))\right)\in \mathcal{N}(\Omega_t)$, and $\beta_0=1$, $\beta_s \left(w + s v\right) \in \mathcal{N}(\Omega)$.
\end{lemma}
\begin{proof}
	At first, we obtain the existence of the function $\alpha_{(\cdot)}$.
	Consider the function $F: (0,+\infty)\times(-\delta,\delta) \to \mathbb{R}$ defined by 
	$$
	F(\alpha,t) = \alpha^{p} \int_{\Omega_t} |D \left(w(\Phi_t^{-1}(y)) + t v(\Phi_t^{-1}(y))\right)|^p \, dy 
	- 
	\alpha^q \int_{\Omega_t} |w(\Phi_t^{-1}(y)) + t v(\Phi_t^{-1}(y))|^q \, dy.
	$$
	Trivially, $F(\cdot,t)$ is differentiable on $(0, +\infty)$ for any $t \in (-\delta,\delta)$. 
	Moreover, $F(\alpha, \cdot)$ is differentiable on $(-\delta, \delta)$ for any $\alpha > 0$.
	Indeed, performing the change of variables, we get
	$$
	F(\alpha,t) = \alpha^{p} \int_{\Omega} |D \left(w + t v\right) \cdot \Psi_t|^p \varphi_t \, dx 
	- 
	\alpha^q \int_{\Omega} |w + t v|^q \varphi_t \, dx,
	$$
	and hence the claimed differentiability follows easily. 
	Since $w \in \mathcal{N}(\Omega)$, we have $F(1,0) = 0$ and
	$$
	F_\alpha'(1,0) = p \int_{\Omega} |D w|^p \, dx - q \int_{\Omega} |w|^q \, dx 
	= 
	(p-q) \int_{\Omega} |w|^q \, dx < 0.
	$$
	Hence, taking $\delta > 0$ smaller (if necessary), we apply the implicit function theorem to deduce the existence of a differentiable function $\alpha_{(\cdot)}: (-\delta,\delta) \to (0,+\infty)$ such that $\alpha_0 = 1$ and $F(\alpha_t,t) = 0$ for all $t \in (-\delta,\delta)$. Noting that the latter equality reads as $\alpha_t \left(w(\Phi_t^{-1}(\cdot)) + t v(\Phi_t^{-1}(\cdot))\right)\in \mathcal{N}(\Omega_t)$, we complete the proof.  
	
	The existence of the function $\beta_{(\cdot)}$ can be obtained arguing along the same lines as above by applying the implicit function theorem to the function $G: (0,+\infty)\times(-\delta,\delta) \to \mathbb{R}$ defined by 
	$$
	G(\beta,s) = \beta^{p} \int_{\Omega} |D \left(w + s v\right)|^p \, dx - \beta^q \int_{\Omega} |w + s v|^q \, dx
	$$
	in a neighbourhood of the point $(1,0)$.
\end{proof}

Finally, to establish the assumption \ref{as:4}, we directly calculate that
\begin{equation}\label{eq:uuneq0}
\left< u, u\right>_0 = E_0''[u](u,u) = (p-1) \int_\Omega |D u|^p \,dx - (q-1) \int_\Omega |u|^q \,dx = (p-q) \int_\Omega |u|^q \,dx \neq 0
\end{equation}
since $p \neq q$,
where we used the fact that $u \in \mathcal{N}(\Omega)$.

\medskip
The main result of this section is the following proposition. 
\begin{proposition}\label{prop:lane-emden}
	Let $u \in \wolp(\Omega)$ be a least energy critical point of $E_0$. 
	Assume that $\dot m(0) = 0$. 
	If a corrector $v \in \wolp(\Omega)$ satisfies $\int_{\Omega} |u|^{q-2} u v \,dx = 0$ and $\left< v, v \right>_0 = -Q[v] > 0$, then
	\begin{align}
	\notag
	\ddot m(0) 
	&= 
	\frac{1}{p}\int_{\Omega} |Du|^{p} \,{\rm div}(\widetilde{R}) \,dx 
	- \frac{1}{q}\int_{\Omega} |u|^{q} \, {\rm div}(\widetilde{R}) \,dx
	-\int_{\Omega} |Du|^{p-2} (Du,  D u \cdot D\widetilde{R}) \,dx
	\\
	\notag
	&+
	\frac{2}{p} \int_{\Omega} |Du|^p \chi_2(DR) \,dx 
	- 
	\frac{2}{q} \int_{\Omega} |u|^q \chi_2(DR) \,dx 
	\\
	\notag
	&-2 \int_{\Omega} |Du|^{p-2}(Du, D u \cdot DR) \, {\rm div}(R) \,dx 
	+ 2 \int_{\Omega} |Du|^{p-2}(Du, D u \cdot DR \cdot DR) \,dx
	\\
	\notag
	& + \int_{\Omega} |Du|^{p-2}(D u \cdot DR, D u \cdot DR) \,dx 
	+ (p-2)\int_{\Omega} |Du|^{p-4}(D u, D u \cdot DR)^2 \,dx 
	\\
	\label{eq:ddotm-plaplacian}
	&+ 
	\frac{(q-p)}{q^2}\frac {\left(\int_{\Omega} |u|^{q} \, {\rm div}(R) \,dx\right)^2}{\int_{\Omega} |u|^{q} \,dx} 
	-\frac { (Q[v])^2 } {\left< v,v \right>_0},
	\end{align}
	where $\chi_2(DR)$ is defined in \eqref{eq:chi2}.
\end{proposition}

Let us discuss how Proposition \ref{prop:lane-emden} follows from Proposition \ref{thm:main} and provide expressions for $\left< v,v \right>_0$ and $Q[v]$. 
First, we observe from \eqref{eq:<>0} and the fact that $u$ is a critical point of $E_0$ that the assumption $\left< u,v \right>_0 = 0$ reads as
$$
\left< u,v \right>_0 
=
(p-1) \int_{\Omega} |Du|^{p-2} (D u, D v) \,dx  - (q-1) \int_{\Omega} |u|^{q-2} u v \,dx
=
(p-q) \int_{\Omega} |u|^{q-2} u v \,dx = 0.
$$
Moreover, $\left< u,u \right>_0$ is given by \eqref{eq:uuneq0} and $\left< v,v \right>_0$ is written in the following way:
\begin{equation*}
\left< v, v \right>_0 
= 
\int_{\Omega} |Du|^{p-2} |D v|^2 \,dx 
+
(p-2) \int_{\Omega} |Du|^{p-4} (Du, D v)^2 \,dx 
- 
(q-1)\int_{\Omega} |u|^{q-2} v^2 \,dx.
\end{equation*}
Second, the functional $Q$ is given by
\begin{align*}
Q[h] 
= 
&-\int_{\Omega} |Du|^{p-2} (D h,  D u \cdot DR) \,dx 
-(p-2)\int_{\Omega} |Du|^{p-4} (Du, D h) (Du,  D u \cdot DR) \,dx 
\\
\notag
&- \int_{\Omega} |Du|^{p-2} (D u,D h \cdot DR) \,dx 
+ \int_{\Omega} |Du|^{p-2} (D u,D h) \, {\rm div}(R) \,dx 
- \int_{\Omega} |u|^{q-2} u h \, {\rm div}(R) \,dx.
\end{align*}
In particular, we get
\begin{equation}\label{eq:Q[u]sim}
Q[u] 
= 
-p\int_{\Omega} |Du|^{p-2} (Du,  D u \cdot DR) \,dx 
+ \int_{\Omega} |Du|^{p} \, {\rm div}(R) \,dx 
- \int_{\Omega} |u|^{q} \, {\rm div}(R) \,dx.
\end{equation}
Let us assume that $\dot m(0) = 0$. 
In view of Proposition \ref{prop:first-derivative}, we have 
\begin{equation}\label{eq:m0emden}
\dot m(0)
=
-\int_{\Omega} |Du|^{p-2} (Du,  D u \cdot DR) \,dx
+
\frac{1}{p}\int_{\Omega} |Du|^{p} \, {\rm div}(R) \,dx 
- \frac{1}{q}\int_{\Omega} |u|^{q} \, {\rm div}(R) \,dx = 0.
\end{equation}
Using \eqref{eq:m0emden}, the expression \eqref{eq:Q[u]sim} can be simplified as 
\begin{align*}
Q[u] 
= 
\frac{p-q}{q}\int_{\Omega} |u|^{q} \, {\rm div}(R) \,dx.
\end{align*}
Combining all these expressions, we directly obtain Proposition \ref{prop:lane-emden} from Proposition \ref{thm:main}. 

\begin{remark}\label{rem:qu=0}
	If the vector field $R$ is divergence-free, then we get $Q[u]=0$. 
	Consequently, if the problem \eqref{eq:optimal_corrector} possesses a solution $w$ such that $Q[w] < 0$, then $w$ is the optimal corrector.
\end{remark}

\subsection{Eigenvalue problem}\label{sec:special_eigenvalue}
In this section, we establish a second-order estimate for the first eigenvalue of the $p$-Laplacian
\begin{equation}\label{eq:lambda1}
\lambda_1(\Omega_t) 
=
\inf_{w \in \wolp(\Omega_t) \setminus \{0\}} \frac{\int_{\Omega_t} |D w|^p \,dy}{\int_{\Omega_t} |w|^p \,dy}.
\end{equation}
Although the functional $\lambda_1(\Omega_t)$ does not directly fall within the assumptions of Section \ref{sec:main}, we will discuss how the arguments of Sections \ref{sec:first-order} and \ref{sec:second_derivative} can be modified to cover this case.

Denote by $u \in \wolp(\Omega)$ a minimizer of $\lambda_1(\Omega)$, that is, the first eigenfunction of the $p$-Laplacian. Note that $u$ is unique up to scaling, see, e.g., \cite{Lind}.
As in Section \ref{sec:main}, consider an admissible function for the minimization problem \eqref{eq:lambda1} of the form 
$$
U_t(y) = u(\Phi_t^{-1}(y)) + t v(\Phi_t^{-1}(y)), \quad y \in \Omega_t,
$$
and its transposition to $\Omega$, $u_t(x) = u(x) + t v(x)$, $x \in \Omega$, where $v \in \wolp(\Omega)$ is an arbitrary corrector. 
Let us denote 
$$
\nu(t) = \frac{\int_{\Omega_t} |D U_t|^p \,dy}{\int_{\Omega_t} |U_t|^p \,dy} = 
\frac{\int_{\Omega} |D u_t \cdot \Psi_t|^p \varphi_t \,dx}{\int_{\Omega} |u_t|^p \varphi_t \,dx}.
$$ 
We see that $\lambda_1(\Omega) = \nu(0)$ and $\lambda_1(\Omega_t) \leq \nu(t)$ for all $|t|<\delta$. 
Thus, as in Section \ref{sec:main}, Taylor's theorem applied to $\nu(t)$ yields
\begin{equation*}
\lambda(\Omega_t) \leq \lambda(\Omega) + \dot{\nu}(0) t + \frac{\ddot{\nu}(0)}{2} t^2 + o(t^2).
\end{equation*}

First, we present an expression for $\dot{\nu}(0)$.
\begin{proposition}\label{prop:eigen-first}
	Let $u \in \wolp(\Omega)$ be a minimizer of $\lambda_1(\Omega)$. 
	Then
	\begin{align}
	\notag
	\dot \nu(0)
	= 
	\frac{p}{\int_{\Omega} |u|^p \,dx} 
	\bigg[
	\frac{1}{p}\int_{\Omega} |Du|^{p} \, {\rm div}(R) \,dx 
	&-\frac{\lambda_1(\Omega)}{p}\int_{\Omega} |u|^{p} \, {\rm div}(R) \,dx
	\\
	\label{eq:m0eigen}
	&-\int_{\Omega} |Du|^{p-2} (Du,  D u \cdot DR) \,dx
	\bigg] = \frac{Q[u]}{\int_\Omega |u|^p \,dx}.
	\end{align} 
\end{proposition}
\begin{remark}\label{rem:eigen}
	Assume that $\Omega$ is of class $C^{1,\zeta}$, $\zeta \in (0,1)$. 
	Then $u \in C^{1,\kappa}(\overline{\Omega})$ for some $\kappa \in (0,1)$, see \cite{lieberman}.
	Therefore, the Pohozaev identity (see, e.g., \cite[Lemma 2]{degiovanni}) applied to \eqref{eq:m0eigen} yields
	\begin{align*}
	\dot \nu(0) 
	=
	-\frac{(p-1)}{\int_{\Omega} |u|^p \,dx} \int_{\partial \Omega} |Du|^p (R,n) \,d\sigma,
	\end{align*}
	where $n$ is the outward unit normal vector to $\partial \Omega$, cf.\ \cite{garcia,lamberti}.
\end{remark}

Second, assuming $\dot{\nu}(0)=0$, we give the following expression for $\ddot{\nu}(0)$.
\begin{proposition}\label{prop:eigen}
	Let $u \in \wolp(\Omega)$ be a minimizer of $\lambda_1(\Omega)$. 
	Assume that $\dot \nu(0) = 0$. 
	If a corrector $v \in \wolp(\Omega)$ satisfies $\left< v, v \right>_0 = -Q[v] > 0$, then
	\begin{align*}
	\notag
	\ddot{\nu}(0) 
	= 
	\frac{p}{\int_{\Omega} |u|^p \,dx} 
	&\bigg[
	\frac{1}{p}\int_{\Omega} |Du|^{p} \, {\rm div}(\widetilde{R}) \,dx 
	-\frac{\lambda_1(\Omega)}{p}\int_{\Omega} |u|^{p} \, {\rm div}(\widetilde{R}) \,dx
	-\int_{\Omega} |Du|^{p-2} (Du,  D u \cdot D\widetilde{R}) \,dx
	\\
	\notag
	&+
	\frac{2}{p} \int_{\Omega} |Du|^p \chi_2(DR) \,dx 
	- 
	\frac{2\lambda_1(\Omega)}{p} \int_{\Omega} |u|^p \chi_2(DR) \,dx 
	\\
	\notag
	&-2 \int_{\Omega} |Du|^{p-2}(Du, D u \cdot DR) \, {\rm div}(R) \,dx 
	+ 2 \int_{\Omega} |Du|^{p-2}(Du, D u \cdot DR \cdot DR) \,dx
	\\
	\notag
	& + \int_{\Omega} |Du|^{p-2}(D u \cdot DR, D u \cdot DR) \,dx 
	+ (p-2)\int_{\Omega} |Du|^{p-4}(D u, D u \cdot DR)^2 \,dx 
	\\
	&
	-\frac { (Q[v])^2 } {\left< v,v \right>_0}
	\bigg],
	\end{align*}
	where $\chi_2(DR)$ is defined in \eqref{eq:chi2}, the functional $Q[v]$ is given by
	\begin{align*}
	Q[v] 
	= 
	&-\int_{\Omega} |Du|^{p-2} (D v,  D u \cdot DR) \,dx 
	-(p-2)\int_{\Omega} |Du|^{p-4} (Du, D v) (Du,  D u \cdot DR) \,dx 
	\\
	&- \int_{\Omega} |Du|^{p-2} (D u,D v \cdot DR) \,dx 
	+ \int_{\Omega} |Du|^{p-2} (D u,D v) \, {\rm div}(R) \,dx 
	\\
	&- \lambda_1(\Omega)\int_{\Omega} |u|^{p-2} u v \, {\rm div}(R) \,dx,
	\end{align*}
	and $\left< v,v \right>_0$ is written as
	\begin{align*}
	\left< v, v \right>_0 
	&= 
	\int_{\Omega} |Du|^{p-2} |D v|^2 \,dx 
	+
	(p-2) \int_{\Omega} |Du|^{p-4} (Du, D v)^2 \,dx 
	- 
	(p-1)\lambda_1(\Omega)\int_{\Omega} |u|^{p-2} v^2 \,dx.
	\end{align*}
\end{proposition}
\begin{remark}\label{rem:eigen-zero}
	Let us explicitly mention that
	$$
	\left<u,h\right>_0 = (p-1) \left(\int_\Omega |D u|^{p-2} (Du, Dh)\,dx - \lambda_1(\Omega) \int_\Omega |u|^{p-2}  uh \,dx\right) = 0
	~\text{ for any }~ h \in \wolp(\Omega),
	$$
	since $u$ is the first eigenfunction of the $p$-Laplacian.
\end{remark}

In the linear case $p=2$, the expressions in Proposition \ref{prop:eigen} can be simplified. 
Note first that
\begin{equation}\label{eq:vv-eigen}
\left< v, v \right>_0 
= 
\int_{\Omega} |D v|^2 \,dx 
- 
\lambda_1(\Omega)\int_{\Omega} v^2 \,dx.
\end{equation}
Recalling that the first eigenfunction $u$ is unique up to scaling, we see from the definition of $\lambda_1(\Omega)$ that $\left< v, v \right>_0 > 0$ for any $v \in \wol2(\Omega) \setminus \{\mathbb{R} u\}$. Thus, we have the following result.
\begin{proposition}\label{lem:eigen}
	Let $p=2$ and $u \in \wol2(\Omega)$ be a minimizer of $\lambda_1(\Omega)$. 
	Assume that $\dot \nu(0) = 0$. 
	If a corrector $v \in \wol2(\Omega) \setminus \{\mathbb{R} u\}$ satisfies $\left< v, v \right>_0 = -Q[v]$, then
	\begin{align*}
	\notag
	\ddot{\nu}(0) 
	= 
	\frac{2}{\int_{\Omega} |u|^2 \,dx} 
	\bigg[
	&\frac{1}{2}\int_{\Omega} |Du|^{2} \, {\rm div}(\widetilde{R}) \,dx 
	-\frac{\lambda_1(\Omega)}{2}\int_{\Omega} |u|^{2} \, {\rm div}(\widetilde{R}) \,dx
	-\int_{\Omega} (Du,  D u \cdot D\widetilde{R}) \,dx
	\\
	&+
	\int_{\Omega} |Du|^2 \chi_2(DR) \,dx 
	- 
	\lambda_1(\Omega) \int_{\Omega} |u|^2 \chi_2(DR) \,dx 
	\\
	\notag
	&- \int_{\Omega} (Du, D u \cdot DR) \, {\rm div}(R) \,dx 
	+  \int_{\Omega} (Du, D u \cdot DR \cdot DR) \,dx
	\\
	&+ \int_{\Omega} (D u \cdot DR, D u \cdot DR) \,dx 
	-\frac { (Q[v])^2 } {\left< v,v \right>_0}
	\bigg],
	\end{align*}
	where $\left<v,v\right>_0$ is given by \eqref{eq:vv-eigen} and the functional $Q[v]$ is written as
	\begin{align}
	\notag
	Q[v] 
	= 
	&-\int_{\Omega} (D v, D u \cdot DR) \,dx 
	- \int_{\Omega} (D u,D v \cdot DR) \,dx 
	+ \int_{\Omega} (D u,D v) \, {\rm div}(R) \,dx 
	\\
	\label{eq:Q[v]-eigen}
	&- \lambda_1(\Omega)\int_{\Omega} u v \, {\rm div}(R) \,dx.
	\end{align}
\end{proposition}

Let us now discuss the proofs of Propositions \ref{prop:eigen-first} and \ref{prop:eigen}. 
Consider the following energy functional acting on $U_t$:
$$
\widetilde{m}(t) = E_t[U_t] = \frac{1}{p} \int_\Omega |D u_t \cdot \Psi_t|^p \varphi_t \,dx 
-
\frac{\nu(t)}{p} \int_\Omega |u_t|^p \varphi_t \,dx.
$$
By definition of $\nu(t)$, we have $\widetilde{m}(t) = 0$ for all $|t| < \delta$, which yields $\dot{\widetilde{m}}(t)= 0$ and $\ddot{\widetilde{m}}(t)=0$. 
The arguments of Sections \ref{sec:first-order}, \ref{sec:second_derivative}, and \ref{sec:special_nonhom} can be applied in much the same way to the functional $\widetilde{m}(t)$ by taking $\alpha_t=1$ and hence $\dot{\alpha}_t=0$ for all $t$.
Therefore, resolving $\dot{\widetilde{m}}(0)= 0$ with respect to $\dot{\nu}(0)$, we obtain Proposition \ref{prop:eigen-first}.
Under the assumption $\dot{\nu}(0) = 0$, the part of $\ddot{\widetilde{m}}(0)$ where the derivatives fall on the integral terms is exactly the same as in \eqref{eq:ddotm-plaplacian} of Proposition \ref{prop:lane-emden} with $p=q$.  
Thus, expressing $\ddot{\nu}(0)$ from the equation $\ddot{\widetilde{m}}(0)=0$, we derive Proposition \ref{prop:eigen}.

\section{Special cases of deformations}\label{sec:lower_dim}
In this section, we present some simplifications of the expressions for $\ddot{m}(0)$ given by Propositions \ref{thm:main} and \ref{prop:lane-emden}, and for $\ddot{\nu}(0)$ given by Propositions \ref{prop:eigen} and \ref{lem:eigen} under the additional assumption that $N=2$ or a vector field $R$ is effectively one-dimensional, i.e., $R = e \rho(x_1,\ldots, x_N)$, where $e$ is a constant vector and $\rho \in C^1(\mathbb{R}^N,\mathbb{R})$ is a scalar function.

\subsection{Effectively one-dimensional deformation}
\begin{lemma}\label{lem:one-dim}
Let $R = (\rho,0,\ldots, 0)^T$, where $\rho \in C^1(\mathbb{R}^N,\mathbb{R})$. Then $DR \cdot DR = {\rm div}(R)\, DR$ and $\chi_2(DR) = 0$.
\end{lemma}
\begin{proof}
Obviously,
\begin{equation*}
DR = 
\begin{pmatrix}
\rho_{x_1} & \ldots & \rho_{x_N}\\
0 & \ldots & 0 \\
\vdots & \ddots & \vdots\\
0 & \ldots & 0
\end{pmatrix},
\quad
\text{and hence}
\quad
DR \cdot DR = 
\begin{pmatrix}
\rho_{x_1}^2 & \rho_{x_1} \rho_{x_2} & \ldots & \rho_{x_1} \rho_{x_N}\\
0 & 0 & \ldots & 0 \\
\vdots & \vdots & \ddots & \vdots\\
0 & 0 & \ldots & 0
\end{pmatrix}.
\end{equation*}
Then, observing that ${\rm div} (R) = \rho_{x_1}$, we have $DR \cdot DR = \rho_{x_1} DR = {\rm div}(R)\, DR$.
\end{proof}

Using Lemma \ref{lem:one-dim}, Proposition \ref{thm:main} can be simplified as follows. 
\begin{proposition}
	Assume that $\widetilde{R} = 0$. Then, under the assumptions of Proposition \ref{thm:main} and Lemma \ref{lem:one-dim}, there holds
	\begin{equation*}
	\ddot m(0) 
	=
	\int_{\Omega} (D^2_{zz} L(u, D u ),D u \cdot DR \otimes D u \cdot DR) \,dx 
	- \frac {(Q[u])^2}{\left< u,u \right>_0} 
	-\frac { (Q[v])^2 } {  \left< v,v \right>_0}.
	\end{equation*}
\end{proposition}

In the case of the Lane-Emden problem, we have the following result.
\begin{proposition}
	Assume that $\widetilde{R} = 0$. Then, under the assumptions of Proposition \ref{prop:lane-emden} and Lemma \ref{lem:one-dim}, there holds
	\begin{align*}
	\notag
	\ddot m(0) 
	&= 
	\int_{\Omega} |Du|^{p-2}(D u \cdot DR, D u \cdot DR) \,dx 
	\\
	&+ (p-2)\int_{\Omega} |Du|^{p-4}(D u, D u \cdot DR)^2 \,dx 
	+\frac{(q-p)}{q^2}
	\frac {\left(\int_{\Omega} |u|^{q} \, {\rm div}(R) \,dx\right)^2}{\int_{\Omega} |u|^{q} \,dx}
	-\frac { (Q[v])^2 } {  \left< v,v \right>_0}.
	\end{align*}
\end{proposition}

Eigenvalue problems covered by Propositions \ref{prop:eigen} and \ref{lem:eigen} can be simplified as follows.
\begin{proposition}
	Assume that $\widetilde{R} = 0$. Then, under the assumptions of Proposition \ref{prop:eigen} and Lemma \ref{lem:one-dim}, there holds
	\begin{align*}
	\ddot{\nu}(0) 
	\notag
	= 
	\frac{p}{\int_{\Omega} |u|^p \,dx} 
	\bigg[
	&\int_{\Omega} |Du|^{p-2}(D u \cdot DR, D u \cdot DR) \,dx 
	\\
	&+ (p-2)\int_{\Omega} |Du|^{p-4}(D u, D u \cdot DR)^2 \,dx 
	-\frac { (Q[v])^2 } {\left< v,v \right>_0}
	\bigg].
	\end{align*}
	If, in addition, $p=2$ and the assumptions of Proposition \ref{lem:eigen} are satisfied, then 
	\begin{equation}\label{eq:first-eigenvalue-1}
	\ddot{\nu}(0) 
	= 
	\frac{2}{\int_{\Omega} |u|^2 \,dx} 
	\bigg[
	\int_{\Omega} (D u \cdot DR, D u \cdot DR) \,dx 
	-\frac { (Q[v])^2 } {\left< v,v \right>_0}
	\bigg].
	\end{equation}	
\end{proposition}

\subsection{Two-dimensional case}
\begin{lemma}\label{lema:two-dim}
Let $N=2$. Then $DR \cdot DR - {\rm div}(R) DR = - {\det}(DR) I$ and $\chi_2(DR) = {\rm det}(DR)$, where $I$ is the identity matrix.
\end{lemma}
\begin{proof}
Both equalities can be proved by direct calculations.
\end{proof}

By means of Lemma \ref{lema:two-dim}, we have the following result on the simplification of Proposition \ref{thm:main}.
\begin{proposition}
	Assume that $\widetilde{R} = 0$. Then, under the assumptions of Proposition \ref{thm:main} and Lemma \ref{lema:two-dim}, there holds
	\begin{align}
	\notag
	\ddot m(0) 
	=
	&-2 \int_{\Omega} (D_z L(u, D u ), D u) {\rm det}(DR) \,dx 
	+
	2 \int_{\Omega} L(u, D u ) {\rm det}(DR) \,dx
	\\
	\notag
	&
	+ \int_{\Omega} (D^2_{zz} L(u, D u ),D u \cdot DR \otimes D u \cdot DR) \,dx 
	- \frac {(Q[u])^2}{\left< u,u \right>_0} 
	-\frac { (Q[v])^2 } {  \left< v,v \right>_0}.
	\end{align}
\end{proposition}

The Lane-Emden problem covered by Proposition \ref{prop:lane-emden} can be simplified as follows.
\begin{proposition}
	Assume that $\widetilde{R} = 0$. Then, under the assumptions of Proposition \ref{prop:lane-emden} and Lemma \ref{lema:two-dim}, there holds
	\begin{align*}
	\notag
	\ddot m(0) 
	= 
	&-\frac{2(p-1)}{p}\int_\Omega |Du|^p {\rm det}(DR) \,dx - \frac{2}{q} \int_\Omega |u|^q {\rm det}(DR) \,dx\\
	&+\int_{\Omega} |Du|^{p-2}(D u \cdot DR, D u \cdot DR) \,dx 
	+ (p-2)\int_{\Omega} |Du|^{p-4}(D u, D u \cdot DR)^2 \,dx 
	\\
	\notag
	&+
	\frac{(q-p)}{q^2}\frac {\left(\int_{\Omega} |u|^{q} \,{\rm div}(R) \,dx\right)^2}{\int_{\Omega} |u|^{q} \,dx}
	-\frac { (Q[v])^2 } {\left< v,v \right>_0}.
	\end{align*}
\end{proposition}

Finally, in the case of the $p$-Laplacian eigenvalue problem, we have the following result.
\begin{proposition}
	Assume that $\widetilde{R} = 0$. Then, under the assumptions of Proposition \ref{prop:eigen} and Lemma \ref{lema:two-dim}, there holds
	\begin{align*}
	\ddot{\nu}(0) 
	\notag
	= 
	\frac{p}{\int_{\Omega} |u|^p \,dx} 
	\bigg[
	&-\frac{2(p-1)}{p}\int_\Omega |Du|^p {\rm det}(DR) \,dx - \frac{2\lambda_1(\Omega)}{p} \int_\Omega |u|^p {\rm det}(DR) \,dx\\
	&+\int_{\Omega} |Du|^{p-2}(D u \cdot DR, D u \cdot DR) \,dx 
	+ (p-2)\int_{\Omega} |Du|^{p-4}(D u, D u \cdot DR)^2 \,dx 
	\\
	&-\frac { (Q[v])^2 } {\left< v,v \right>_0}
	\bigg].
	\end{align*}
\end{proposition}

\section{Applications}\label{sec:applications}
In this section, we provide an application of our results to the stability of the first eigenvalue $\lambda_1(\Omega)$ of the Laplacian in the rectangle $\Omega = (0,1) \times (-a,a)$ with $a \geq 1$ under perturbations of the form $\Phi_t(x,y) = (x,y)^T + tR(x,y)$, where
$R(x,y) = (f(x) \theta(y),0)^T$ and $f,\theta$ are $C^1$-smooth, see, e.g., Figures \ref{fig:theta=y} and \ref{fig:theta=siny} below. 
That is, such $R$ satisfies Lemma \ref{lem:one-dim}. 
In some cases, we are able to find explicitly an optimal corrector, which gives us a possibility to compare $\ddot{\nu}(0)$ computed for optimal and several nonoptimal correctors. 

First, we note that in view of the separable nature of our domain we have
$$
\lambda_1(\Omega) = \pi^2 + \left(\frac{\pi}{2a}\right)^2
\quad \text{and} \quad
u = \sin(\pi x) \cos \left(\frac{\pi y}{2a}\right).
$$
In order to use the formula \eqref{eq:first-eigenvalue-1} for $\ddot{\nu}(0)$, we have to require  $\dot{\nu}(0) = 0$. Applying the Pohozaev identity (cf.\ Remark \ref{rem:eigen}), we see that 
$$
\dot{\nu}(0) = -\frac{1}{\int_\Omega |u|^2 \,dx} \int_{\partial \Omega} |Du|^2 (R,n) \, d\sigma
=
-\frac{\pi^2 (f(1)-f(0))}{\int_\Omega |u|^2 \,dx} \int_{-a}^{a} \left|\cos\left(\frac{\pi y}{2a}\right)\right|^2 \theta(y) \,dy.
$$
Therefore, if either $f(0)=f(1)$ or $\theta(y)$ is odd, then $\dot{\nu}(0) = 0$. Below, we will work with the case of odd  $\theta(y)$.

For the considered $R$, we have
\begin{align*}
\text{div}(R) = f'(x) \theta(y)
\quad \text{and} \quad 
DR = 
\left(
\begin{matrix}
f'(x) \theta(y)	& f(x)\theta'(y) \\
0				& 0
\end{matrix}
\right).
\end{align*}
Thus,
\begin{align*}
\int_{\Omega} \left(D u \cdot DR, D u \cdot DR\right) \,dx
&=
\int_{\Omega} u_x^2 (f' \theta)^2 \,dx + 
\int_{\Omega} u_x^2 (f \theta')^2 \,dx,
\end{align*}
and  we conclude from \eqref{eq:first-eigenvalue-1} that
\begin{equation}\label{eq:m0fin1}
\ddot{\nu}(0) 
= 
\frac{2}{\int_{\Omega} |u|^2 \,dx} 
\left[
\int_{\Omega} u_x^2 (f' \theta)^2 \,dx + 
\int_{\Omega} u_x^2 (f \theta')^2 \,dx
- \frac{(Q[v])^2}{\left<v,v\right>_0}
\right],
\end{equation}
where a corrector $v \in \wol2(\Omega) \setminus \{\mathbb{R} u\}$ satisfies $\left<v, v\right>_0 = -Q[v]$, and $Q[v]$ and $\left<v, v\right>_0$ are given by \eqref{eq:Q[v]-eigen} and \eqref{eq:vv-eigen}, respectively, i.e.,
\begin{equation}\label{eq:num_Q}
Q[v] 
= 
-\int_\Omega u_x v_x f' \theta \,dx
-\int_\Omega u_x v_y f \theta' \,dx
-\int_\Omega u_y v_x f \theta' \,dx
+\int_\Omega u_y v_y f' \theta \,dx
-\lambda_1(\Omega) \int_\Omega u v f' \theta \,dx,
\end{equation}
\begin{equation}\label{eq:num_vv}
\left<v, v\right>_0 
=
\int_\Omega |Dv|^2 \,dx - \lambda_1(\Omega) \int_\Omega |v|^2 \,dx.
\end{equation}

Let us consider the problem of finding the optimal corrector (see Section \ref{sec:optimal-bvp}):
\begin{equation*}
\lambda = \sup\left\{\frac{(Q[v])^2}{\left<v,v\right>_0}:~ v \in \wol2(\Omega) \setminus \{0\},~ \left<u,v\right>_0 = 0 
\right\}.
\end{equation*}
Note that the constraint $\left<u,v\right>_0 = 0$ is fulfilled for any $v \in  \wol2(\Omega)$ due to Remark \ref{rem:eigen-zero}.
Moreover, since $u,f,\theta$ are regular, we easily see that the Rayleigh quotient $\frac{(Q[v])^2}{\left<v,v\right>_0}$ is weakly upper semicontinuous and
$$
\frac{(Q[v])^2}{\left<v,v\right>_0} \leq C \frac{\int_\Omega |Dv|^2 \,dx}{\int_\Omega |Dv|^2 \,dx - \lambda_1(\Omega) \int_\Omega |v|^2 \,dx},
$$
where $C=C(u,f,\theta)$ does not depend on $v$. 
Let $\{v_n\}$ be a maximizing sequence for $\lambda$. We can always assume $Q[v_n]=1$ for all $n \in \mathbb{N}$.
Moreover, $\{v_n\}$ does not converge to $u$ weakly in $\wol2(\Omega)$ since $Q[u]=0$ by $\dot{\nu}(0)=0$ and Proposition \ref{prop:eigen-first} Therefore, we conclude that there exists a maximizer $v$ of $\lambda$, and $Q[v]=1$.
We see from Lemma \ref{lem:optimal} that $v$ satisfies
\begin{equation}\label{eq:w-app}
\left<v, h\right>_0 = \lambda^{-1} Q[h]
~\text{ for all } ~
h \in \wol2(\Omega).
\end{equation}
Making the substitution $w = -\lambda v$, we conclude that, for the optimal corrector $w$,
\begin{equation*}
\ddot{\nu}(0) 
= 
\frac{2}{\int_{\Omega} |u|^2 \,dx} 
\left[
\int_{\Omega} u_x^2 (f' \theta)^2 \,dx + 
\int_{\Omega} u_x^2 (f \theta')^2 \,dx
- \left<w,w\right>_0
\right].
\end{equation*}
Performing an integration by parts in \eqref{eq:w-app}, we see that the optimal corrector $w$ is the solution of the following boundary value problem:
\begin{equation}\label{eq:w}
\left\{
\begin{aligned}
-\Delta w - \lambda_1 w 
&= 
2u_{xx} f' \theta + 2 u_{xy} f \theta' + u_x(f''\theta + f\theta'') \quad \text{in } \Omega,\\
w&=0 \quad \text{on } \partial \Omega.
\end{aligned}
\right.
\end{equation}
The solution $w$ of \eqref{eq:w} can be expressed via the Fourier series as
$$
w = \sum_{(m,k)\neq(1,1)} v_{m,k} \varphi_{m,k},
$$
where $m,k \in \mathbb{N}$, and 
$$
v_{m,k} = \frac{1}{\lambda_{m,k} - \lambda_{1,1}} \int_\Omega \left(2u_{xx} f' \theta + 2 u_{xy} f \theta' + u_x(f''\theta + f\theta'')\right) \varphi_{m,k} \,dx.
$$
Here, the eigenvalues $\lambda_{m,k}$ and the eigenfunctions $\varphi_{m,k}$ are given, respectively, by
$$
\lambda_{m,k} = m^2 \pi^2 + \left(\frac{k \pi}{2a}\right)^2
\quad \text{and} \quad
\varphi_{m,k} 
=
\left\{
\begin{aligned}
&\sqrt{\frac{2}{a}}\sin\left(m \pi x\right) \cos\left(\frac{k \pi y}{2a}\right) \quad \text{for } k \text{ odd},\\
&\sqrt{\frac{2}{a}}\sin\left(m \pi x\right) \sin\left(\frac{k \pi y}{2a}\right)  \quad \text{for } k \text{ even}.
\end{aligned}
\right.
$$
Hence, in view of \eqref{eq:w}, we have
\begin{align}
\notag
\left<w,w\right>_0 
&= 
\int_\Omega \left(2u_{xx} f' \theta + 2 u_{xy} f \theta' + u_x(f''\theta + f\theta'')\right) w \,dx
\\ 
\label{eq:num_ww}
&= 
\sum_{(m,k)\neq(1,1)} \frac{1}{\lambda_{m,k} - \lambda_{1,1}} 
\left(
\int_\Omega \left(2u_{xx} f' \theta + 2 u_{xy} f \theta' + u_x(f''\theta + f\theta'')\right) \varphi_{m,k} \,dx
\right)^2.
\end{align}

Now we are ready to consider several explicit examples of the perturbation $R$. We will treat the following six cases (see Figures \ref{fig:theta=y} and \ref{fig:theta=siny}):
\begin{enumerate}[label={\rm(\roman*)}]
	\item\label{case:sinxy} $f(x)=\sin\left(\frac{\pi x}{2}\right)$ and $\theta(y)=y$;
	\item\label{case:xy} $f(x)=x$ and $\theta(y)=y$;
	\item\label{case:cosxy} $f(x)=1-\cos\left(\frac{\pi x}{2}\right)$ and $\theta(y)=y$;
	\item\label{case:sinxsiny} $f(x)=\sin\left(\frac{\pi x}{2}\right)$ and $\theta(y)=\sin\left(\frac{\pi y}{2a}\right)$;
	\item\label{case:xsiny} $f(x)=x$ and $\theta(y)=\sin\left(\frac{\pi y}{2a}\right)$;
	\item\label{case:cosxsiny} $f(x)=1-\cos\left(\frac{\pi x}{2}\right)$ and $\theta(y)=\sin\left(\frac{\pi y}{2a}\right)$.
\end{enumerate}

For all cases \ref{case:sinxy}-\ref{case:cosxsiny} we consider nonoptimal correctors $v = yu$ and $v = \varphi_{1,2}$, as well as several approximations of the optimal corrector $w$:
$$
w_{M,K} = \sum_{(m,k)\neq(1,1), m \leq M, k \leq K} v_{m,k} \varphi_{m,k}.
$$
Moreover, in the cases \ref{case:sinxsiny} and \ref{case:xsiny}, we obtain an analytic expression for the sum in $\left<w,w\right>_0$ and hence compute $\ddot{\nu}(0)$ for the optimal corrector.
In the cases \ref{case:sinxy}-\ref{case:cosxy}, \ref{case:cosxsiny},  we are not able to obtain an analytic expression for the sum in $\left<w,w\right>_0$, that is why only approximations of $w$ will be considered. 
Notice that each integral in \eqref{eq:m0fin1}, \eqref{eq:num_Q}, \eqref{eq:num_vv} can be easily calculated analytically for such choices of $v$ (although the resulting expressions are relatively huge).
We omit trivial calculations and only discuss the behaviour of $\ddot{\nu}(0)$ with respect to $a \geq 1$.

First, we consider the cases \ref{case:sinxy}-\ref{case:cosxy}. The behaviour of $\ddot{\nu}(0)$ with respect to $a \in (1,1.1)$ is depicted on Figures \ref{fig:sinxy}, \ref{fig:xy}, and \ref{fig:cosxy}.
We observe that $\ddot{\nu}(0) < 0$ for all sufficiently large $a > 1$, which implies that $\lambda(\Omega_t)$ decays locally with respect to $t$ for such values of $a$.
In particular, in each case, $\ddot{\nu}(0) < 0$ for all $a \geq 1.01$ when $v=w_{4,6}$.
Note that we do not expect $\ddot{\nu}(0) = 0$ at $a=1$, since in this case the shape of the right boundary of the deformed domain does not coincide with the shape of the nodal line of any second eigenfunction in the square $(0,2)\times(-1,1)$.
That is, if $a=1$, then $\lambda_1(\Omega_t)$ has to increase with respect to $t$.
Note also that $f(x)=1-\cos\left(\frac{\pi x}{2}\right)$ gives better values for $\ddot{\nu}(0)$ for, at least, $v=yu$ and $v=\varphi_{1,2}$. This can be explained by the fact that the mass of $u$ near the left boundary $x=0$ changes slower with respect to this deformation than with respect to $f(x)=\sin\left(\frac{\pi x}{2}\right)$ and $f(x)=x$.

\begin{figure}[!h]
	\centering
	\begin{minipage}[t]{0.49\linewidth}
		\centering
		\includegraphics[width=0.35\linewidth]{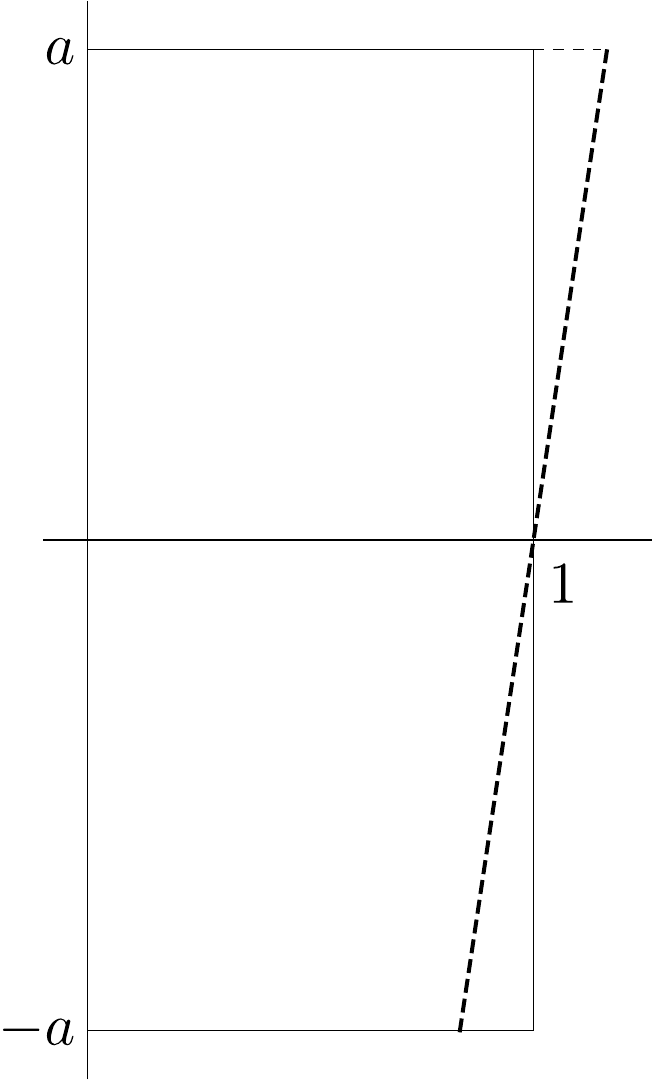}
		\caption{Perturbation by $\theta(y)=y$}
		\label{fig:theta=y}
	\end{minipage} 	
	\begin{minipage}[t]{0.49\linewidth}
		\centering
		\includegraphics[width=0.95\linewidth]{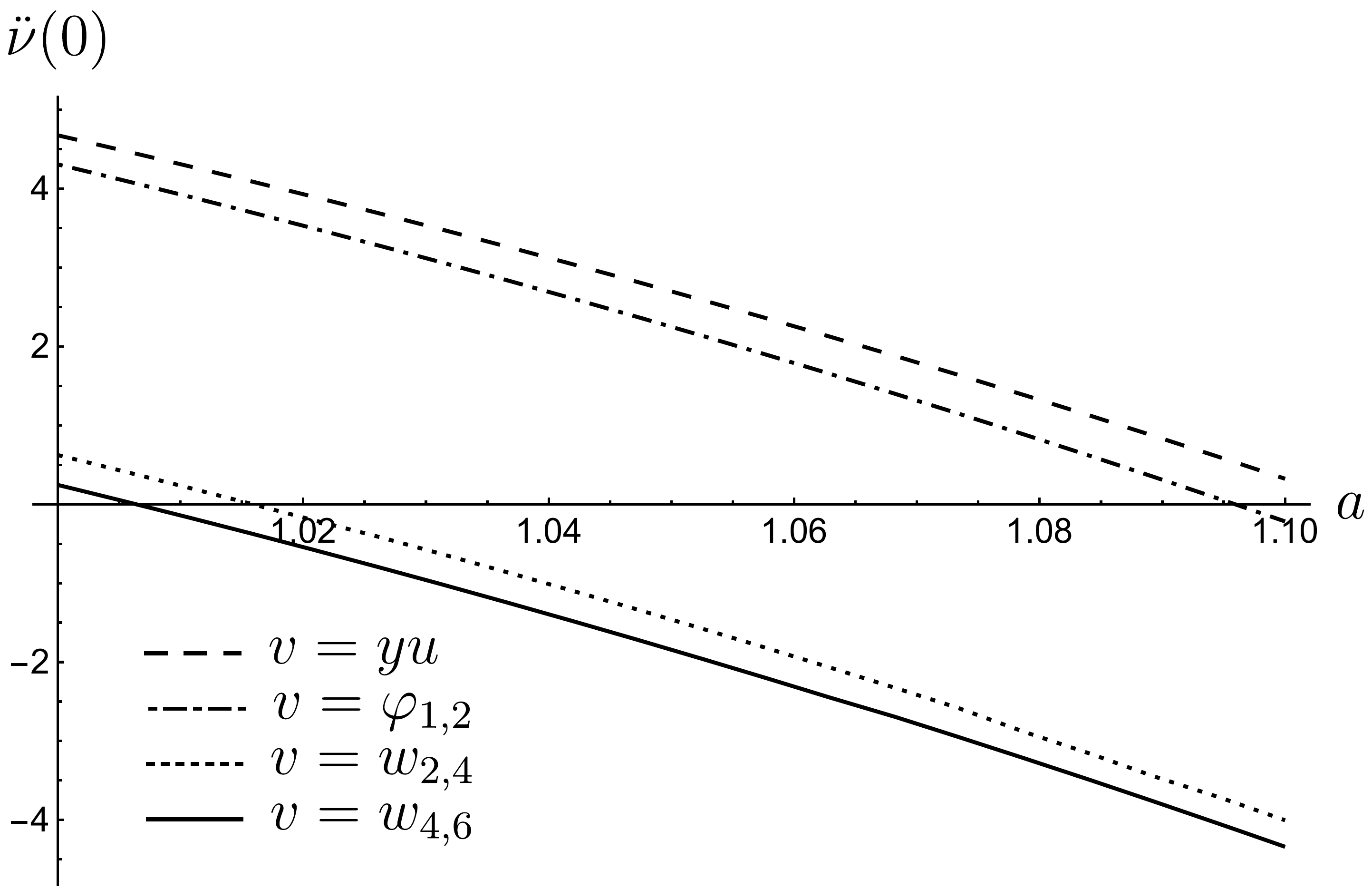}
		\caption{$f(x)=\sin\left(\frac{\pi x}{2}\right)$ and $\theta(y)=y$}
		\label{fig:sinxy}
	\end{minipage} 	
\end{figure} 
\begin{figure}[!h]
	\centering
	\begin{minipage}[t]{0.49\linewidth}
		\centering
		\includegraphics[width=0.95\linewidth]{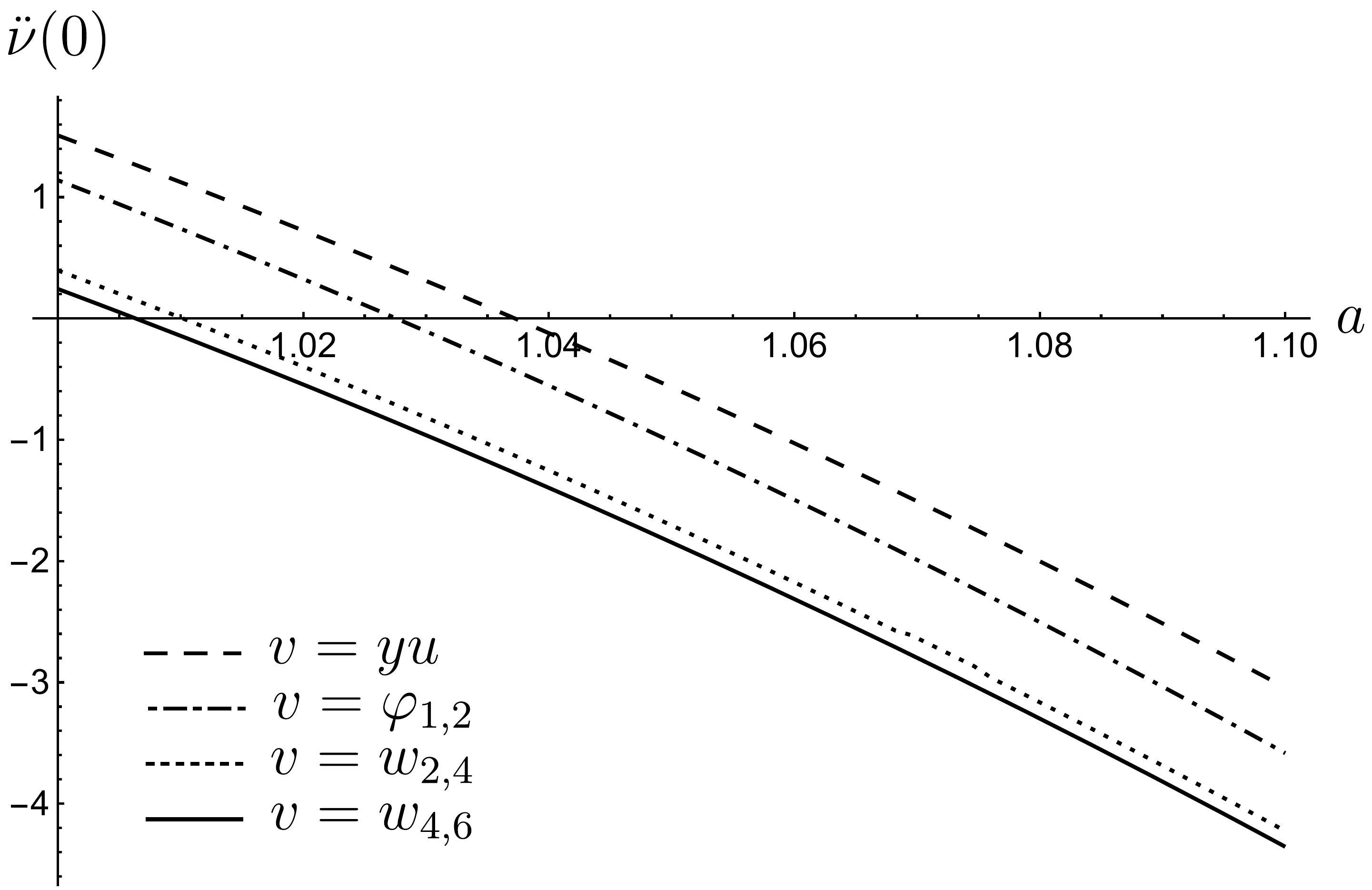}
		\caption{$f(x)=x$ and $\theta(y)=y$}
		\label{fig:xy}
	\end{minipage} 	
	\begin{minipage}[t]{0.49\linewidth}
		\centering
		\includegraphics[width=0.95\linewidth]{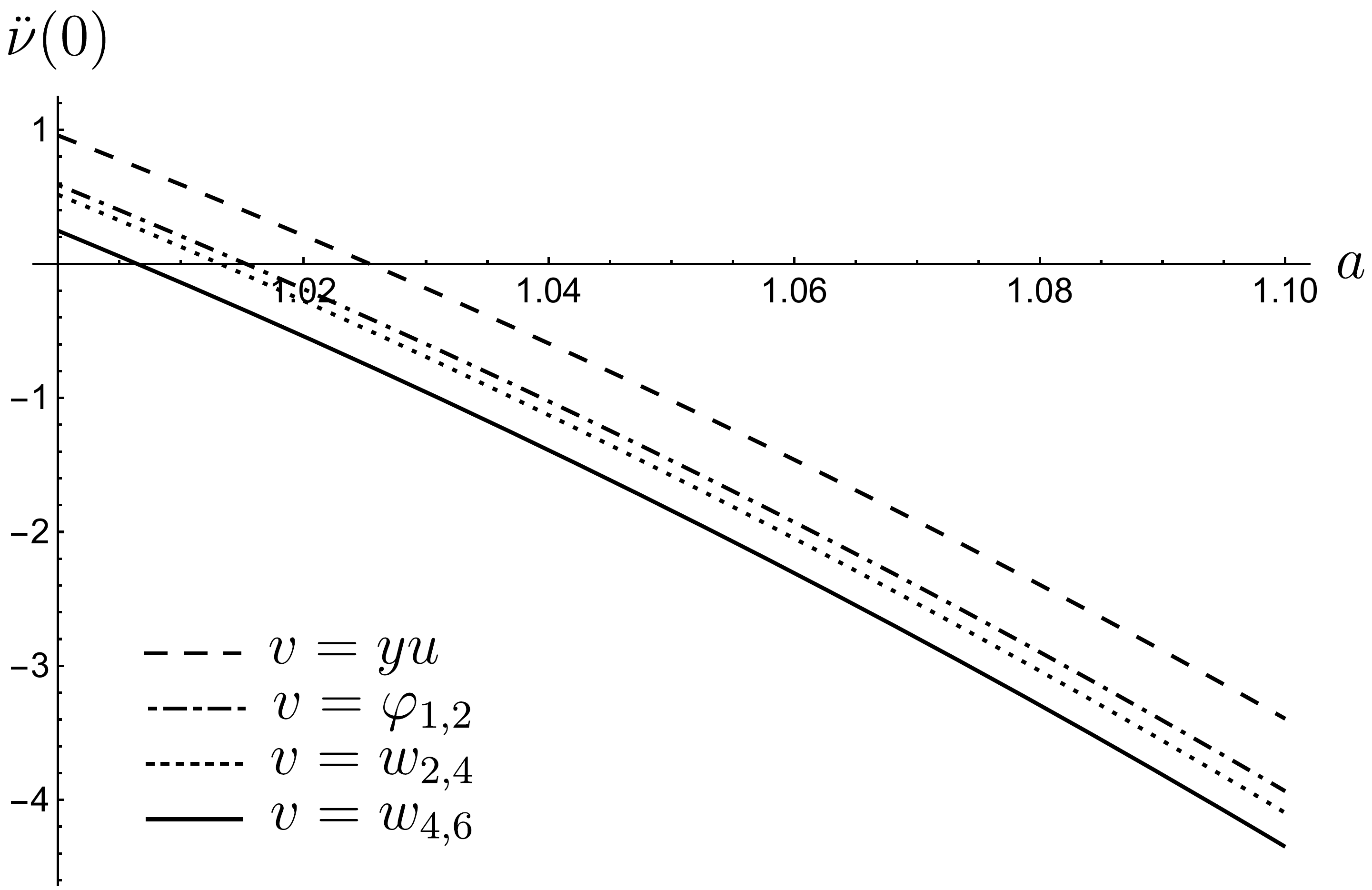}
		\caption{$f(x)=1-\cos\left(\frac{\pi x}{2}\right)$ and $\theta(y)=y$}
		\label{fig:cosxy}
	\end{minipage} 	
\end{figure}

Second, we consider the cases \ref{case:sinxsiny}-\ref{case:cosxsiny}. 
The behaviour of $\ddot{\nu}(0)$ with respect to $a \in (1,1.1)$ is depicted on Figures \ref{fig:sinxsiny}, \ref{fig:xsiny}, and \ref{fig:cosxsiny}.
In the case $f(x)=\sin\left(\frac{\pi x}{2}\right)$, we have 
$$
\int_{\Omega} u_x^2 (f' \theta)^2 \,dx + 
\int_{\Omega} u_x^2 (f \theta')^2 \,dx
=
\frac{\pi^4}{64}\left(a+\frac{3}{a}\right).
$$
Moreover, calculating the integrals in \eqref{eq:num_ww} and using \textsl{Mathematica{\tiny$^{\textregistered}$}} to find an analytic expression for the sum over $m$ and $k$, we get
$$
\left<w,w\right>_0 = 
\frac{\pi^3}{64 a}\left(3\pi + \pi a^2 -8 a \sqrt{4a^2-3} \cot\left(\frac{\pi \sqrt{4a^2-3}}{2a}\right)\right).
$$
In the case $f(x)=x$, we have
$$
\int_{\Omega} u_x^2 (f' \theta)^2 \,dx + 
\int_{\Omega} u_x^2 (f \theta')^2 \,dx
=
\frac{\pi^2(2\pi^2+8a^2+3)}{64a}
$$
and, applying \textsl{Mathematica$^{\tiny{\textregistered}}$} again, we get
$$
\left<w,w\right>_0 = 
\frac{\pi^2}{64 a}\left(3+8a^2+2\pi^2-8\pi a \sqrt{4a^2-3} \cot\left(\frac{\pi \sqrt{4a^2-3}}{2a}\right)\right).
$$
Thus, the expressions for $\ddot{\nu}(0)$ with optimal correctors are obtained analytically for all $a \geq 1$ for these choices of $f$. It is not hard to see that the corresponding values of $\ddot{\nu}(0)$ for these $f$'s coincide for any $a \geq 1$. 
This observation reflects the fact that the second-order shape variation of $\lambda_1(\Omega)$ depends only on the perturbation of the boundary, and does not depend on how the perturbation acts inside the domain, see \cite{novruzipierre}.

We again observe that $\ddot{\nu}(0) < 0$ for all sufficiently large $a > 1$ and any choice of $f$ and the corrector $v$.
For instance, $\ddot{\nu}(0) < 0$ for all $a \geq 1.01$ by choosing the nonoptimal corrector $v=w_{2,2}$. 
Moreover, in the cases \ref{case:sinxsiny} and \ref{case:xsiny}, $\ddot{\nu}(0) < 0$ for all $a > 1$ and $\ddot{\nu}(0)=0$ for $a=1$, when the optimal corrector $v=w$ is considered. This is naturally anticipated, since for $a=1$, the perturbation driven by $\theta(y)=\sin\left(\frac{\pi y}{2 a}\right)$ changes the right boundary according to the behaviour of the nodal set of the second eigenfunction $\varphi_{2,1}+\varepsilon \varphi_{1,2}$ in the square $(0,2)\times(-1,1)$ for sufficiently small $\varepsilon>0$. That is, if $a=1$, then $\lambda_1(\Omega_t)$ will be unchanged.
We also see that $f(x)=1-\cos\left(\frac{\pi x}{2}\right)$ gives better values of $\ddot{\nu}(0)$ for, at least, $v=yu$ and $v=\varphi_{1,2}$. 

\begin{figure}[!h]
	\centering
	\begin{minipage}[t]{0.49\linewidth}
		\centering
		\includegraphics[width=0.35\linewidth]{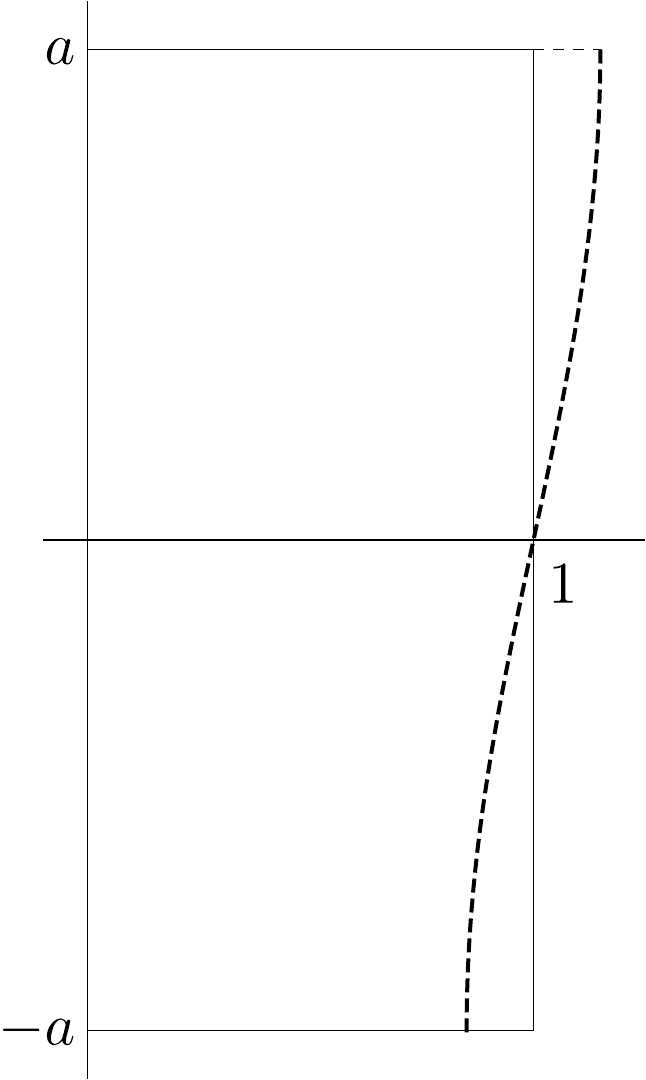}
		\caption{Perturbation by $\theta(y)=\sin\left(\frac{\pi y}{2 a}\right)$}
		\label{fig:theta=siny}
	\end{minipage} 	
	\begin{minipage}[t]{0.49\linewidth}
		\centering
		\includegraphics[width=0.95\linewidth]{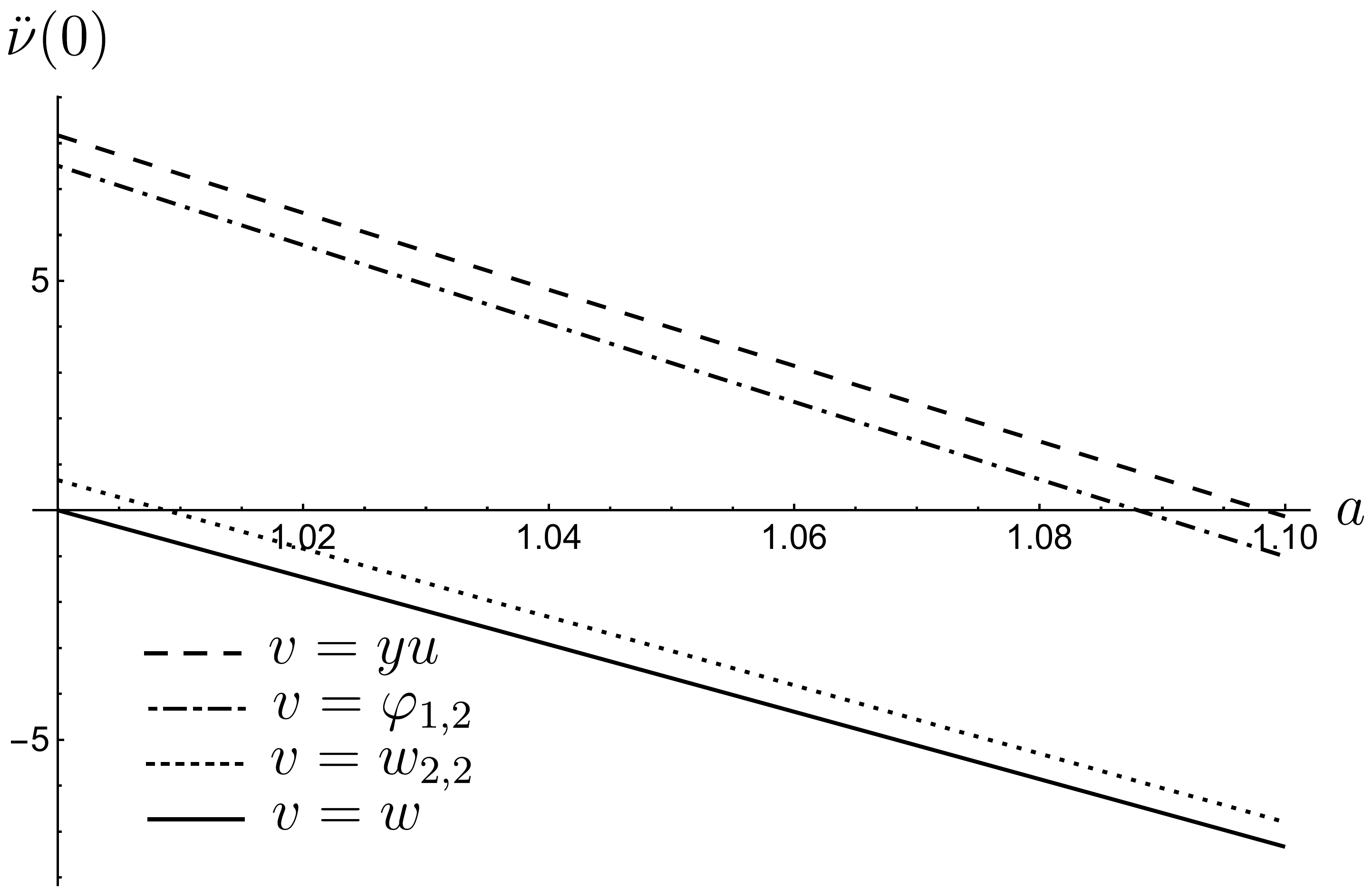}
		\caption{$f(x)=\sin\left(\frac{\pi x}{2}\right)$ and $\theta(y)=\sin\left(\frac{\pi y}{2 a}\right)$}
		\label{fig:sinxsiny}
	\end{minipage} 	
\end{figure} 
\begin{figure}[!h]
	\centering
	\begin{minipage}[t]{0.49\linewidth}
		\centering
		\includegraphics[width=0.95\linewidth]{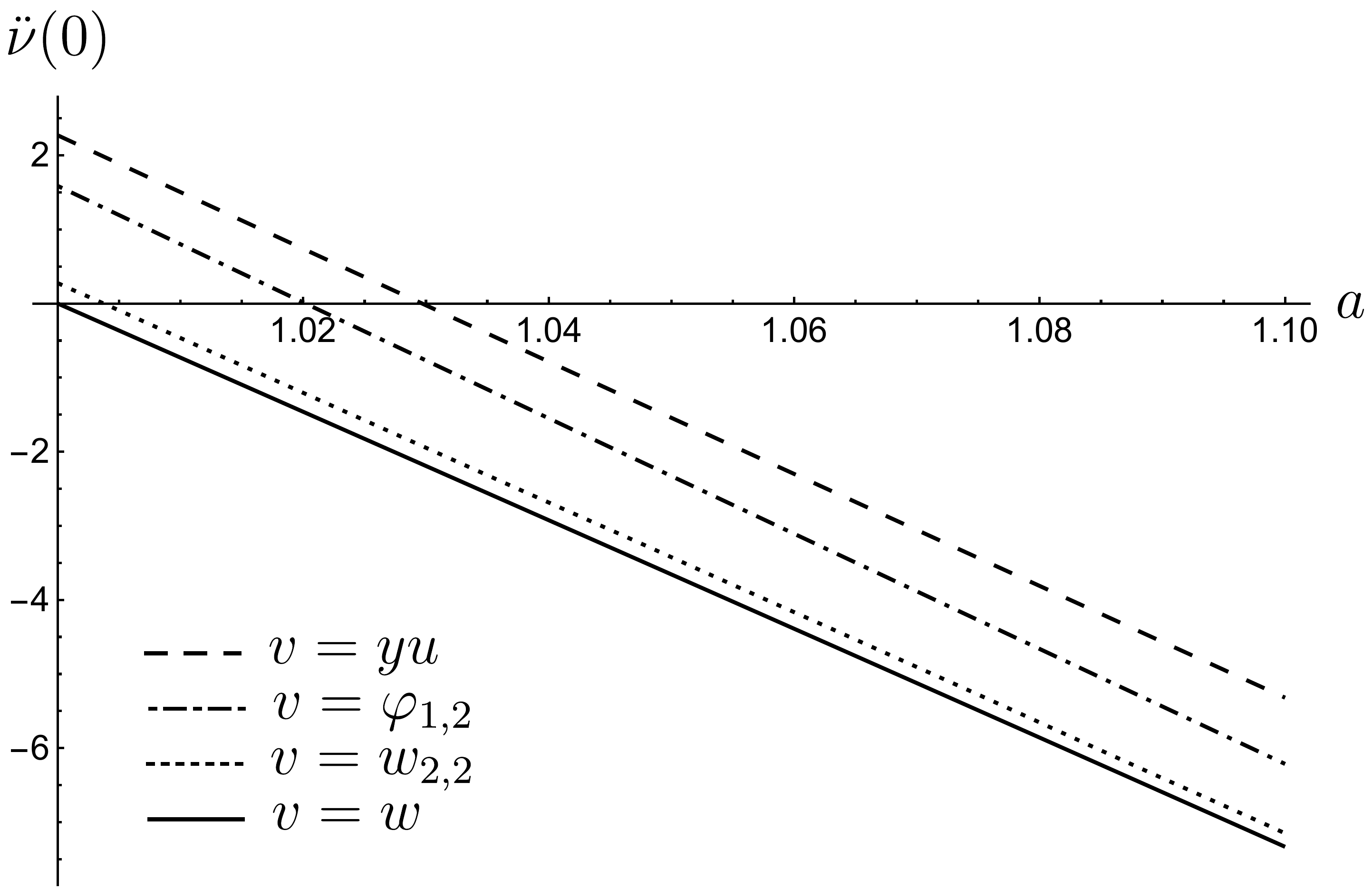}
		\caption{$f(x)=x$ and $\theta(y)=\sin\left(\frac{\pi y}{2 a}\right)$}
		\label{fig:xsiny}
	\end{minipage} 	
	\begin{minipage}[t]{0.49\linewidth}
		\centering
		\includegraphics[width=0.95\linewidth]{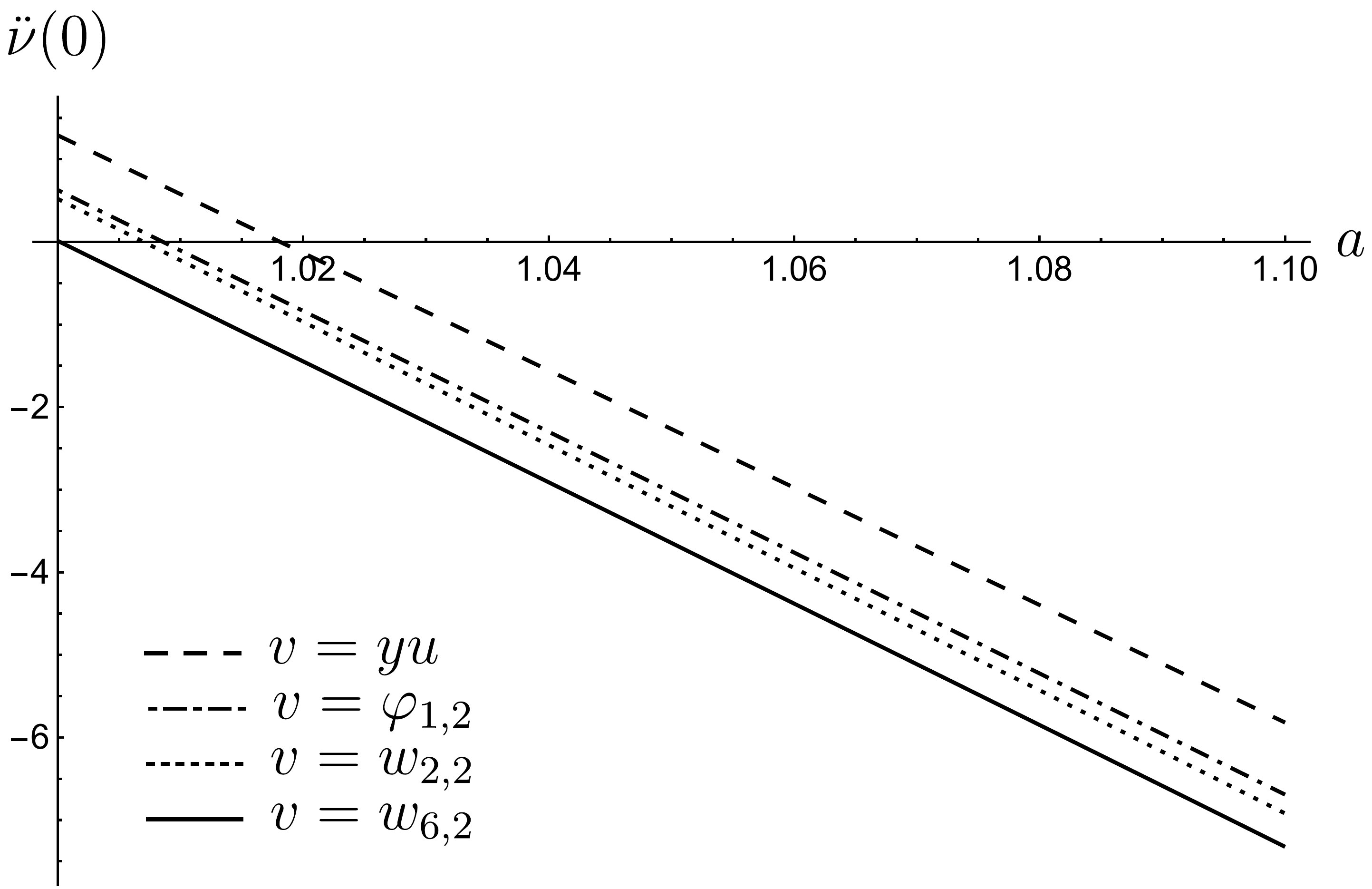}
		\caption{$f(x)=1-\cos\left(\frac{\pi x}{2}\right)$ and $\theta(y)=\sin\left(\frac{\pi y}{2 a}\right)$}
		\label{fig:cosxsiny}
	\end{minipage} 	
\end{figure}

\bigskip
\noindent
{\bf Acknowledgements.}
The first author was supported by the project LO1506 of the Czech Ministry of Education, Youth and Sports, and by the grant 18-03253S of the Grant Agency of the Czech Republic. 
The second author was supported by the grant 17-01-00678 of Russian Foundation for Basic Research.
The second author wishes to thank the University of West Bohemia, where this research was started, for the invitation and hospitality.
The authors would like to thank A.I.\ Nazarov for stimulating discussions and valuable advices. 
Moreover, the authors are grateful to the anonymous referee whose suggestions and remarks led to the substantial improvement of the manuscript.

\end{document}